\documentclass[12pt, reqno]{amsart}
\usepackage{amsfonts}
\usepackage{amssymb}
\usepackage{amsthm}
\usepackage{amsmath}
\usepackage{enumitem}

\usepackage{graphicx}
\newtheorem{thm}{Theorem}[section]

\newtheorem{lemma}[thm]{Lemma}

\newtheorem{example}{Example}

\newtheorem{rmk}[thm]{Remark}
\newtheorem{defn}[thm]{Definition}
\newtheorem{prop}[thm]{Proposition}

\numberwithin{equation}{section}
\begin{document}
\title{Genus bounds for min-max minimal surfaces}
\author{Daniel Ketover}\address{Imperial College London\\Huxley Building, 180 Queen's Gate, London SW7 2RH}
\thanks{The author was partially supported by NSF-PRF DMS-1401996 as well as by ERC-2011-StG-278940.}
 \email{d.ketover@imperial.ac.uk}
\maketitle
\begin{abstract}
We prove optimal genus bounds for minimal surfaces arising from the min-max construction of Simon-Smith.  This confirms a conjecture made by Pitts-Rubinstein in 1986.
\end{abstract}

\section{Introduction}
A Heegaard surface in a closed oriented $3$-manifold $M$ is a closed orientable surface $\Sigma$ such that $M\setminus\Sigma$ consists of two open genus $g$ handlebodies.  It is a fundamental fact that every $3$-manifold admits a Heegaard surface of some genus.  Given such a Heegaard surface $\Sigma$ in a $3$-manifold endowed with a Riemannian metric, a natural question is whether one can find a minimal surface in the isotopy class of $\Sigma$.  Certainly it would be fruitless to try to minimize area because one can always collapse $\Sigma$ into either of the two 1-dimensional spines of the handlebodies determined by the Heegaard splitting so that its area approaches zero.

Alternatively, one can try to find not a minimizing but an index $1$ saddle-point type critical point for the area functional in the isotopy class of $\Sigma$ using a min-max argument.    The min-max argument was introduced by Birkhoff \cite{B} in 1917 to find a closed geodesic on the two-sphere endowed with an arbitrary metric.   For surfaces, the method was pioneered by Almgren-Pitts (\cite{a},\cite{p}) who used it to produce embedded minimal surfaces in arbitrary $3$-manifolds.  

We will consider the min-max process of Simon-Smith \cite{ss}, that refined the work of Almgren-Pitts.  Thus we consider sweepouts of a $3$-manifold by surfaces isotopic to a Heegaard surface $\Sigma$ except where they degenerate to $1$-d graphs in the spines of both handlebodies.  Among such Heegaard sweepouts (or heuristically,``cycles" in the space of surfaces), by analogy with finite dimensional Morse theory, the largest slice of a sweepout with smallest possible maximal area should be a critical point for the area functional, i.e.\ potentially a smooth minimal surface.  

More precisely, given a Heegaard splitting $H$ of $M$, a \emph{sweepout by Heegaard surfaces} or \emph{sweepout} is a one parameter family of closed sets $\left\{\Sigma_t\right\}_{t\in [0,1]}$ continuous in the Hausdorff topology such that 
\begin{enumerate}
\setlength{\itemsep}{1pt}
  \setlength{\parskip}{0pt}
  \setlength{\parsep}{0pt}

\item $\Sigma_t$ is an embedded smooth surface isotopic to $H$ for $t\in(0,1)$
\item $\Sigma_t$ varies smoothly for $t\in (0,1)$
\item $\Sigma_{0}$ and $\Sigma_1$ are $1$-d graphs, each one a spine of one of the handlebodies determined by the splitting surface $H$.
\end{enumerate}

If $\Lambda$ is a collection of sweepouts, we say that the set $\Lambda$ is \emph{saturated} if given a map $\phi\in C^{\infty}(I\times M,M)$ such that $\phi(t,-)\in\text{Diff}_{0} M$ for all $t\in I$, and a family  $\left\{\Sigma_t\right\}_{t\in I}\in\Lambda$, we have  $\left\{\phi(t,\Sigma_t)	\right\}_{t\in I}\in\Lambda$.  Given a Heegaard splitting $H$, denote by $F_H$ a sweepout of $M$ by Heegaard surfaces that is also a foliation of $M\setminus\{\Sigma_0,\Sigma_1\}$.  Denote by $\Lambda_H$ the smallest saturated family of sweepouts containing $F_H$.

The width associated to $\Lambda_H$ is defined to be
\begin{equation}\label{w}
W(M,\Lambda_H)=\inf_{\left\{\Sigma_t\right\}\in\Lambda}\sup_{t\in I} \mathcal{H}^2(\Sigma_t),
\end{equation}
where $\mathcal{H}^2$ denotes $2$-dimensional Hausdorff measure.  It follows by an easy argument using the isoperimetric inequality (Proposition 1.4 in \cite{cd}) that $W_H>0$.  This expresses the non-triviality of the sweepout.  A \emph{minimizing sequence} is a sequence of families $\left\{\Sigma^n_t\right\}\in\Lambda_H$ such that
\begin{equation} 
\lim_{n\rightarrow\infty}\sup_{t\in[0,1]} \mathcal{H}^2(\Sigma^n_t)=W(M,\Lambda_H).
\end{equation}
\
A \emph{min-max sequence} is then a sequence of slices $\Sigma^n_{t_n}$, $t_n\in (0,1)$ such that
\begin{equation} 
\mathcal{H}^2(\Sigma^n_{t_n})\rightarrow W(M,\Lambda_H).
\end{equation}

The main result due to Simon-Smith is that some min-max sequence converges to a smooth minimal surface realizing the width, whose genus is controlled:

\begin{thm}[Simon-Smith Min-Max Theorem (1982) \cite{cd},\cite{dp},\cite{ss}]\label{SimonSmith}
Let $M$ be a closed oriented Riemannian $3$-manifold admitting a Heegaard surface $H$ of genus $g$.  Then 
some min-max sequence $\Sigma_{t_i}^i$ of surfaces isotopic to $H$ converges as varifolds to $\sum_{j=1}^k n_j \Gamma_j$, where $\Gamma_j$ are smooth embedded pairwise disjoint minimal surfaces and where $n_j$ are positive integers.  Moreoever, \begin{equation}
W(M, \Lambda_H)=\sum_{j=1}^k n_j\mathcal{H}^2(\Gamma_j).
\end{equation} The genus of the limiting minimal surface can be controlled as follows:
\begin{equation}\label{gb}
\sum_{i\in O} g(\Gamma_i) + \frac{1}{2}\sum_{i\in N} (g(\Gamma_i)-1)\leq g,
\end{equation}
where $O$ denotes the set of $i$ such that $\Gamma_i$ is orientable, and $N$ the set of $i$ such that $\Gamma_i$ is non-orientable, and $g(\Gamma)$ denotes the genus of $\Gamma$.  The genus of a non-orientable surface is the number of cross-caps one must attach to a two-sphere to obtain a homeomorphic surface.
\end{thm}

Simon-Smith proved Theorem \ref{SimonSmith} in the case when the ambient manifold $M$ is diffeomorphic to $\mathbb{S}^3$, and $g=0$. Thus Simon-Smith showed that the three-sphere endowed with arbitrary metric always contains an embedded minimal two-sphere.   The arguments of Simon-Smith easily generalize to the setting of genus $g$ sweepouts of arbitrary three-manifolds.  The exposition of Colding-De Lellis \cite{cd} and De Lellis-Pellandini \cite{dp} together give a complete account of the Simon-Smith Min-Max Theorem \ref{SimonSmith} in its full generality.

The work of Simon-Smith was based on refining the seminal work of Almgren-Pitts (\cite{a},\cite{p}), who proved the existence of min-max minimal surfaces using more general sweepouts that do not vary as smoothly as those considered here.

Theorem \ref{SimonSmith} is the only known method that works with no assumptions on the ambient metric to produce embedded minimal surfaces.  Alternatively, starting with a Heegaard foliation one can try to run the mean curvature flow on each slice, and hope that some slice will converge to a minimal surface in the limit as time approaches infinity.  Such a result was proved for sweepouts of the two-sphere by circles by Grayson \cite{g} but the mean curvature flow at present is far from being able to achieve this for surfaces sweeping out a three-manifold.

In some ways however, Theorem \ref{SimonSmith} is unsatisfactory.  The genus bound \eqref{gb} is not optimal and does not give any qualitative information about how the min-max limit is achieved topologically from its approximating min-max sequence.  

To see what optimal genus bound one can expect, it is instructive to consider the work of Meeks-Simon-Yau (Theorem 1 in \cite{msy}).  The authors considered a minimizing sequence for area in a non-trivial isotopy class of some surface $\Sigma$ embedded in a three-manifold.   They proved that a minimizing sequence converges to $\sum_{j=1}^k n_j\Gamma_j$ where $\Gamma_j$ are smooth pairwise disjoint minimal surfaces and $n_j$ are positive integers.   They further showed that after performing finitely many neckpinch surgeries on the minimizing sequence and potentially discarding several small components, the remaining components of the minimizing sequence align themselves as covers with the expected multiplicities about the set of limiting $\Gamma_j$.    (Recall that we say a surface $\tilde{\Sigma}$ arises from $\Sigma$ via a \emph{neck-pinch surgery} if $\Sigma\setminus\tilde{\Sigma}$ is an annulus, $\tilde{\Sigma}\setminus\Sigma$ consists of two disks, and the symmetric difference $\tilde{\Sigma}\Delta\Sigma$ is a sphere bounding a ball.)

Pitts-Rubinstein claimed in 1986 \cite{pr} that the same statement about the way in which minimizing sequences can degenerate should apply in the min-max setting.  In this paper we confirm their conjecture.

Let us first introduce some notation. Given a closed embedded surface $\Gamma\subset M$, denote by $T_\epsilon(\Gamma)$ the $\epsilon$-tubular neighborhood about $\Gamma$. For a component $\Gamma_j$ of the min-max limit in Theorem \ref{SimonSmith} occurring with multiplicity $n_j$ let us make the following definitions: if $n_j$ is even, set
\begin{equation}
S^\epsilon_{n_j}(\Gamma_j):=\bigcup_{k=1}^{n_j/2}\partial T_{\epsilon/k}(\Gamma_j),
\end{equation} 
and if $n_j$ is odd, 
\begin{equation}
S^\epsilon_{n_j}(\Gamma_j):=\Gamma_j\cup\bigcup_{k=1}^{(n_j-1)/2}\partial T_{\epsilon/k}(\Gamma_j).
\end{equation}

Thus for orientable $\Gamma_j$, $S^\epsilon_{n_j}(\Gamma_j)$ consists of $n_j$ copies of $\Gamma_j$ stacked about it, while if $\Gamma_j$ is non-orientable and $n_j$ is even, then $S^\epsilon_{n_j}(\Gamma_j)$ consists of $n_j/2$ orientable surfaces of genus $g(\Gamma_j)-1$, each a double cover of $\Gamma_j$ via nearest point projection.

We can now state the main result of this paper:

\begin{thm} [Convergence Theorem for Min-Max Sequences]\label{main}
Let $H$ be a genus $g$ Heegaard surface in a Riemannian $3$-manifold.  By Theorem \ref{SimonSmith}, some choice of min-max sequence $\Sigma_j$ isotopic to $H$ converges as varifolds to $\sum_{i=1}^k n_i\Gamma_i$, where $\Gamma_i$ are smooth closed pairwise disjoint embedded minimal surfaces and $n_i$ are positive integers.  There exists $\epsilon_1>0$ so that for any $0<\epsilon<\epsilon_1$, after passing to a subsequence, for $j$ large enough, by performing finitely many neck-pinch surgeries on $\Sigma_j$ within $T_\epsilon(\cup_{i=1}^k\Gamma_i)$ and discarding some connected components, one obtains a surface $\tilde{\Sigma}_j$ such that
\begin{enumerate}
\setlength{\itemsep}{1pt}
  \setlength{\parskip}{0pt}
  \setlength{\parsep}{0pt}

\item $\tilde{\Sigma}_{j}\subset T_{\epsilon}(\cup_{i=1}^k \Gamma_i)$ 
\item $\tilde{\Sigma}_j$ is isotopic to $\bigcup_{i=1}^k S^\epsilon_{n_i}(\Gamma_i)$.
\end{enumerate}
Moreover, one has the following ``genus bound with multiplicity:"
\begin{equation}\label{bettergenus}
\sum_{i\in O} n_i g(\Gamma_i) + \frac{1}{2}\sum_{i\in N} n_i (g(\Gamma_i)-1)\leq g.
\end{equation}
\end{thm}

\begin{rmk}
The conclusion of Theorem \ref{main} holds as long as the min-max sequence is $1/j$-almost minimizing. See Section \ref{replacement2} for the definition of this property.  
\end{rmk}
\begin{rmk}
When $\Gamma_i$ is non-orientable, $n_i$ is necessarily even since surgery on an orientable surface cannot give rise to a non-orientable connected component.
\end{rmk}
\begin{rmk}
In some situations using the fact that the min-max sequence arises from a Heegaard splitting, the surgery process proved in Theorem \ref{main} implies slightly sharper genus bounds as discussed by Pitts-Rubinstein \cite{pr}.  However, \eqref{bettergenus} is the most one can expect for sequences that are merely almost minimizing (see the discussion in Section 10.2 in \cite{dp}).
\end{rmk}
Note that \eqref{bettergenus} is stronger than \eqref{gb} since it has the multiplicities $n_i$ on the left hand side.  The genus bound \eqref{bettergenus} follows directly from the statement on degenerations  because surgery can only decrease the genus.  Theorem \ref{main} also applies in settings slightly more general than sweepouts arising from Heegaard splittings.  One can consider saturations from the sweepouts consisting of level sets of a Morse function as in Definition 0.5 of \cite{dp}.

The usefulness of Theorem \ref{main} is that it permits one to use min-max theory in topological arguments and therefore is a bridge between the geometric min-max process and three-manifold topology.  Already Theorem \ref{main} has found other applications in topology.  In \cite{CGK} it is used in our resolution of the classification problem for non-Haken hyperbolic $3$-manifolds.   

Pitts-Rubinstein also sketched an argument (Theorem 1.8 in \cite{r}) claiming that if one began with a \emph{strongly irreducible} Heegaard splitting $H$ (as introduced in \cite{CG}), using an iterated min-max process one could prevent any degeneration and produce a minimal surface isotopic to $H$ or else a stable non-orientable minimal surface with multiplicity $2$ double covered by $H$ after performing a single neck-pinch on $H$.  Their sketch uses the results of this paper as well as a Morse index bound. 

Note that Theorem \ref{main} differs slightly from the degeneration for minimizing sequences proved by Meeks-Simon-Yau in that we do not quantify the size of the neck-pinches (which in \cite{msy} are referred to as $\gamma$-reductions).   The arguments of this paper could likely be extended to achieve this but it does not seem to be needed in most applications.

While Theorem \ref{main} tells us \emph{how} min-max sequences may degenerate, the main challenge is to determine \emph{when} and \emph{whether} one can prevent the min-max limit from degenerating at all.  In fact the key step in Marques-Neves' \cite{mn} proof of the Willmore conjecture was to control the genus of the min-max limit associated to a complicated five parameter family of tori in round $\mathbb{S}^3$.  They used a degree argument that is intimately connected with the geometry of $\mathbb{S}^3$.  Together with the assumption of additional symmetry \cite{K}, the Catenoid Estimate \cite{KMN}, and Lusternick-Schnirelman theory, the optimal genus bounds do allow us to control the genus of the limiting minimal surface in many situations.  For instance, in \cite{KMN} we showed that if a three-manifold $M$ has positive Ricci curvature, then running a min-max procedure relative to a Heegaard surface realizing the Heegaard genus of $M$, no degeneration can occur and one obtains via Theorem \ref{main} a minimal surface isotopic to the Heegaard surface.  Recently Marques-Neves \cite{mn2} have also proved some Morse index bounds for min-max limits which should have applications to understanding the limiting geometries.

In brief, we prove Theorem \ref{main} by showing that if the degeneration were not as claimed, then there would have to be a ``folding curve" as in Example \ref{bad} (see below).  But min-max sequences are locally approximated by stable minimal surfaces and limits of stable surfaces do not exhibit such folding behavior by the curvature estimates of Schoen.  To reduce to the problem of ruling out folding, we will take stable replacements for our min-max sequence in a manner somewhat analagous to an iteration of Schwarz' alternating method \cite{Sc}.

The organization of the paper is as follows.  In Section \ref{genuscontrol} we describe previous efforts to control the topology of the limiting surfaces using Simon-Smith's lifting lemma and our improvement.  In Section \ref{basic} we collect relevant facts from the theory of minimal surfaces that we shall need.  In Section \ref{replacement2} we summarize the variational property that min-max sequences satisfy and how it is used to replace them locally with stable minimal surfaces.  In Section \ref{strategy} we give an explanation of the idea for improving the lifting lemma and the rest of Section \ref{improvedlifting2} is devoted to proving this new lifting lemma.  In Section \ref{last} we show how the lifting lemma implies Theorem \ref{main} using an argument of Frohman-Hass \cite{fh}.  
\\
\\
\noindent
{\em Acknowledgements:  I am grateful to Camillo De Lellis for several discussions and for his hospitality in inviting me to Zurich where some of this work was carried out.  I thank my advisor Toby Colding for suggesting this problem and his support.}

\section{Genus Bounds and Lifting Lemma}\label{genuscontrol}
Let us first consider the situation addressed by Simon-Smith \cite{ss}.  One has therefore a min-max sequence consisting of two-spheres, and one wants to show that the varifold limit is a union of two-spheres (possibly with multiplicity).  The difficulty is that it is possible, for instance, for a sequence of two-spheres to converge as varifolds to a torus with multiplicity two:

\begin{example}\label{bad}
Let $\mathbb{T}^3=\mathbb{R}^3/\mathbb{Z}^3$.  Define the two-torus
\begin{equation}
T:=[0,1]\times[0,1]\times\{1/2\}\subset\mathbb{T}^3.
\end{equation}
For each small $\epsilon$ consider the set 
\begin{equation}
R_\epsilon:=[\epsilon,1-\epsilon]\times[\epsilon,1-\epsilon]\times[1/2-\epsilon,1/2+\epsilon]\subset\mathbb{T}^3.
\end{equation}
For each $\epsilon$, $\partial R_\epsilon$ is a closed two-sphere embedded in $\mathbb{T}^3$.  On other hand, $\partial R_\epsilon$ converges as varifolds to $2T$ as $\epsilon\rightarrow 0$.  Of course $\partial R_\epsilon$ are not Heegaard surfaces, but one can easily make an analagous construction to find a sequence of two-spheres converging to a torus with any even multiplicity in $\mathbb{S}^3$ (see Figure 1 in \cite{dp}).
\end{example}

Thus one has to use more than the varifold convergence to rule Example \ref{bad} out.  The essential tool in achieving this and proving Theorem \ref{SimonSmith} is Simon-Smith's lifting lemma:
\begin{lemma} [Simon-Smith's Lifting Lemma \cite{ss,dp}]\label{SimonSmithLifting}
There exists a finite set $\mathcal{P}\subset\Gamma_i$ (where $\Gamma_i$ is a component of the min-max limit occurring with multiplicity $n_i$) so that if $\gamma$ is a simple closed curve on $\Gamma_i\setminus\mathcal{P}$ and for any $\epsilon >0$, when $j$ is large enough, there is a positive $n\leq n_i$ and a closed curve $\tilde{\gamma}$ on $\Sigma_j\cap T_{\epsilon}(\Gamma_i)$ which is homotopic to $n\gamma$ in $T_{\epsilon}(\Gamma_i)$. 
\end{lemma}

To see how Simon-Smith's lifting lemma allows one to exclude the convergence in Example 1, consider the closed curve $\gamma(t)_{t=0}^1 =\{1/2\}\times\{t\}\times\{1/2\}\subset T$ and the putative min-max sequence $\partial R_{1/j}\rightarrow 2T$. Then by the lifting lemma, $\gamma$ would have to ``lift" on $\partial R_{1/j}$ to a curve $\tilde{\gamma}$ homotopic to $\gamma$ or twice $\gamma$ in $T_\epsilon(T)$.  But since $\partial R_{1/j}$ is a two-sphere contained in $T_\epsilon(T)$, $\tilde{\gamma}$ is null-homotopic in $\partial R_{1/j}\cap T_\epsilon(T)$ and thus also in $T_\epsilon(T)$.


To obtain the optimal genus bounds \eqref{bettergenus} one expects that the min-max sequence is arranging itself as coverings about $\Gamma_i$ and thus one needs to be able to replace $n\leq n_i$ in Lemma \ref{SimonSmithLifting} with $n=n_i$. We can now state the improved lifting lemma that achieves this. This is the key technical ingredient in proving Theorem \ref{main}.

First, let us fix notation.  For a closed embedded surface $\Sigma$ in a three-manifold, and $\epsilon$ small enough, let $p$ denote the nearest point projection map $p:T_\epsilon(\Sigma)\rightarrow\Sigma$.  If $\gamma$ is a simple curve contained on $\Sigma$, $p^{-1}(\gamma)$ is either an annulus or M\"obius band.  The latter can only occur when $\Gamma$ is non-orientable.

\begin{prop} [Improved Lifting Lemma with Multiplicity]\label{improvedlifting}

 Let $\Sigma_j$ be a $1/j$-almost minimizing min-max sequence of smooth surfaces arising from a Heegaard splitting (thus by Theorem \ref{SimonSmith}, $\Sigma_j\rightarrow \sum_{i=1}^s n_i\Gamma_i$, where $\Gamma_i$ are smooth embedded pairwise disjoint minimal surfaces and $n_i$ are positive integers).    Let $\left\{\gamma_i\right\}_{i=1}^k$ be a collection of simple closed curves contained in $\Gamma:=\bigcup_{i=1}^s\Gamma_i$.   Assume there is a point $p_i\in\Gamma_i$ in any connected component of $\Gamma$ containing one of the $\gamma_j$ so that $\gamma_r\cap\gamma_s=p$ for all $r\neq s$ with $\gamma_r,\gamma_s\subset\Gamma_i$. Then there exists $\epsilon_0 >0$, so that for any $\epsilon<\epsilon_0$,  there exist curves $\left\{\tilde{\gamma}_i\right\}_{i=1}^k\subset\Gamma$, as well as a subsequence of $\Sigma_j$ (still labeled $\Sigma_j$), and surfaces $\tilde{\Sigma}_j$ obtained from $\Sigma_{j}$ by finitely many neck-pinch surgeries, such that 
\begin{enumerate}
\setlength{\itemsep}{1pt}
  \setlength{\parskip}{0pt}
  \setlength{\parsep}{0pt}

\item Each $\tilde{\gamma}_i$ is homotopic to $\gamma_i$ in $\Gamma$ and $\tilde{\gamma}_i\subset T_{\epsilon}(\gamma_i)$
\item $\tilde{\Sigma}_{j}\rightarrow\sum_{i=1}^s n_i\Gamma_i$ as varifolds
\item For each $i\in\left\{1,2,..k\right\}$, if $\tilde{\gamma}_i\subset\Gamma_l$ then either $p^{-1}(\tilde{\gamma}_i)\cap T_\epsilon(\Gamma_l)\cap\tilde{\Sigma}_j$ is a union of $n_l$ closed curves, each of which projects via $p$ onto $\tilde{\gamma}_i$ with degree one or else $p^{-1}(\tilde{\gamma}_i)\cap T_\epsilon(\Gamma_l)\cap\tilde{\Sigma}_j$ is a union of $n_l/2$ closed curves, each of which projects via $p$ onto $\tilde{\gamma}_i$ with degree two.  The latter option can occur only when $\Gamma_l$ is non-orientable and $p^{-1}(\tilde{\gamma}_i)$ is a M\"obius band.
\end{enumerate}
\end{prop}
\begin{rmk}
The condition that all curves intersect in a point is a technical condition that can certainly be removed.   
\end{rmk}

Our proof of Proposition \ref{improvedlifting}  is logically independent of Simon-Smith's original work, though we do use an auxilliary result on boundary regularity that was proved as Lemma 8.1 in De Lellis-Pellandini \cite{dp}.

Given Proposition \ref{improvedlifting} one can use an argument of Frohman-Hass (Theorem 2.4 in \cite{fh}) to deduce our main result Theorem \ref{main}.  We carry this out in Section \ref{last}.

\noindent
.

\section{Regularity Theory}\label{regularity}
\subsection{Preliminary lemmas and notation}\label{basic}
We will use repeatedly the following estimate due to Schoen \cite{schoen} (cf.\ \cite{W}, \cite{cmestimates}).
\begin{lemma} [Schoen's Curvature Estimates for Stable Surfaces \cite{schoen}\cite{cmestimates}]
Let $U$ be an open set in a 3-manifold, and $\Sigma_i$ a sequence of stable minimal surfaces in $U$ with $\partial\Sigma_i\subset\partial U$.  Then a subsequence of $\Sigma_i$ converges to a stable minimal surface $\Sigma$ smoothly on compact subsets of $U$.  Also for any stable surface $\Gamma\subset U$ with $\partial\Gamma\subset\partial U$, there holds for each $x\in\Gamma$,
\begin{equation}\label{schoen}
|A|^2(x)\leq C_U dist(x,\partial U)^{-2},
\end{equation}
where $C_U$ only depends on $U$ and not the stable surface $\Gamma$. 
\end{lemma}

\begin{rmk}\label{cinf}
The precise meaning of ``converges smoothly on compact sets" is that, for all compact $K\subset U$, for $j$ large enough we can express $K\cap\Sigma_j$ as a union of normal graphs of the form $\left\{exp_x(f_j(x)n(x)): x\in\Sigma\right\}$ where $$||f_j(x)||_{C^k(K)}\rightarrow 0$$ for all k.
\end{rmk}

We shall also need the monotonicity formula, stated here for smooth minimal surfaces.  Given $x\in M$ and $V$ a surface in $M$, define the ratio 
$$f(s) = \frac{\mathcal{H}^2(B_x(s)\cap V)}{\pi s^2}.$$
\begin{lemma} [Monotonicity Formula]\label{mono}
There exists a function $C(r)\geq 1$ (defined for $r$ small enough) such that if $V$ is a smooth minimal surface and $x\in M$,
\begin{equation}
 f(s)\leq C(r)f(t)\mbox{       whenever } 0\leq s\leq t\leq r.
\end{equation}
Moreover, $C(r)$ approaches $1$ as $r\rightarrow 0$.
\end{lemma}

Another useful topological lemma is from \cite{cm2}, which gives that for a sequence of surfaces with bounded genus, the genus can collapse into at most finitely many points:

\begin{lemma} [Genus Collapse, Colding-Minicozzi (Lemma I.0.14.\cite{cm2})]\label{genuscollapse}
Suppose  $\Sigma_j$ is a sequence of smooth surfaces a 3-manifold of genus $g$.  Then there exists finitely many points in the manifold $\left\{x_i\right\}_{i=1}^n$ and a subsequence of the surfaces, still denoted $\Sigma_j$, such that for all $x\notin\left\{x_i\right\}_{i=1}^n$, there is a radius $r(x)$ such that $\Sigma_j\cap B_x(r)$ is a union of planar domains for $r\leq r(x)$.  In particular, $g(\Sigma_j\cap B_x(r))=0$.
\end{lemma}
\begin{rmk}
The genus of a surface with boundary is defined to be the genus of the closed surface obtained by capping off each boundary circle with a disk.
\end{rmk}
\subsection{Theory of Replacements}\label{replacement2}
Because our proof of Theorem \ref{main} makes heavy use of the regularity theory, we explain in this section how the regularity of the min-max limit is proved.   See Colding-De Lellis' survey \cite{cd} or Pitts' original manuscript \cite{p} for more details.  Roughly speaking the point of the replacement theory is that min-max sequences are locally well approximated by stable surfaces and one can use these stable surfaces to better understand the limit.

We begin with the definition of the main variational property that Almgren \cite{a} first discovered for min-max sequences:
\begin{defn}
 Given $\epsilon>0$ an open set $U\subset M^3$ and a surface $\Sigma$ we say that $\Sigma$ is $\epsilon$-almost minimizing in $U$ if there does not exist any isotopy $\phi$ supported in $U$ such that 
\begin{enumerate}
\item $ \mathcal{H}^2(\phi(t,\Sigma))\leq \mathcal{H}^2(\Sigma) + \frac{\epsilon}{8}$
\item  $\mathcal{H}^2(\phi(1,\Sigma))\leq \mathcal{H}^2(\Sigma) - \epsilon$
\end{enumerate}
\end{defn}
\noindent

Almgren was able to show \cite{a} that min-max sequences satisfy this property in small enough annuli:

\begin{thm}[\cite{cd}, \cite{dp}] \label{am}
Given a Heegaard splitting there exists a smooth min-max sequence $\Sigma_j$ of surfaces isotopic to $H$ and a function $r:M\rightarrow R$ such that in every annulus $An$ centered at x and with outer radius at most $r(x)$, $\Sigma_j$ is $1/j$-almost minimizing in $An$ provided $j$ is large enough.  Moreover, $\Sigma_j$ converges to a stationary varifold $V$ as $j\rightarrow\infty$.
\end{thm}

 We will call a sequence of surfaces \emph{$1/j$-almost minimizing} if it satisfies the conclusion of Theorem \ref{am}.  The necessity of working in small annuli rather than small balls is due to the example of the ``three-legged starfish" (cf.\ Figure 6 in \cite{cd} and the Introduction in \cite{p}). The almost-minimizing property of $\Sigma_j$ is the only variational property needed to prove the regularity of $V$: 

\begin{thm}[\cite{cd}, \cite{dp}] \label{am2}
Any sequence of genus $g$ surfaces that is $1/j$-almost minimizing converges to a smooth embedded minimal surface for which the genus bound \eqref{gb} holds.
\end{thm}

Thus Theorem \ref{am} and Theorem \ref{am2} together imply the Simon-Smith Min-Max Theorem \ref{SimonSmith}.

The task in this paper is thus to use the $1/j$-almost minimizing property to prove the stronger genus bounds \eqref{bettergenus}.  We will need first to understand how the almost minimizing property is used in proving the regularity of $V$ in Theorem \ref{am2}.  To that end, let us make the following definition (where $||V||$ denotes the mass of the varifold $V$):

\begin{defn}\label{replacement}
Given a stationary varifold $V$ in $M$ and $U$ an open subset of $M$, $V$ has \emph{a replacement in $U$} if there exists a stationary varifold $V'$ such that 
\begin{enumerate}
\item $V' = V$ in $M\setminus\overline{U}$
\item $||V'|| = ||V||$
\item $V'$ is smooth in $U$.
\end{enumerate}
\end{defn}

The $1/j$ almost minimizing property allows us to construct a replacement for $V$ in any domain $U\subset\bar{U}\subset B_x(r(x))\setminus \{x\}$ as follows.  We need yet another definition: 

\begin{defn}
Let $\mathcal{I}$ be a class of isotopies of $M$ and $\Sigma\subset M$ a smooth embedded surface.   If ${\phi^k}\in\mathcal{I}$ and $$\lim_{k\rightarrow\infty}\mathcal{H}^2(\phi^k(1,\Sigma)) = \inf_{\phi\in\mathcal{I}}\mathcal{H}(\phi(1,\Sigma)),$$
then we will say that $\phi^k(1,\Sigma)$ is a \emph{minimizing sequence for Problem($\Sigma,\mathcal{I}$)}.
\end{defn}
\noindent
Let $Is_j(\Sigma, U)$ denote all isotopies $\phi$ of $M$ supported in $U$ such that 
\begin{equation}
\mathcal{H}^2(\phi(t,\Sigma))\leq \mathcal{H}^2(\Sigma) + \frac{1}{8j} \mbox{ for all }0\leq t\leq 1.
\end{equation}
\noindent
Let us call the isotopies $Is_j(\Sigma, U)$ the \emph{$1/j$ isotopies for $\Sigma$ supported in $U$}.
For each $j$ we then take a minimizing sequence $\phi^l(1,\Sigma_j)$ for Problem($\Sigma_j$, $Is_j(\Sigma_j,U$)) so that
\begin{equation}\label{smooth}
\phi^l(1,\Sigma_j)\rightarrow V_j, \mbox{ where $V_j$ is a stable surface in $U$}.
\end{equation}

We will call $V_j$ a \emph{$1/j$-replacement} for $\Sigma_j$ in $U$.  By Lemma 7.4 of \cite{cd}, $V_j$ is smooth. By the $1/j$-almost minimizing property of $\Sigma_j$, we have 
\begin{equation}\label{almostminimizingprop}
\mathcal{H}^2(V_j)\geq \mathcal{H}^2(\Sigma_j) -1/j.
\end{equation}
\noindent
By Schoen's curvature estimates (\ref{schoen}), $V_j$ converge to a varifold $\hat{V}$ that coincides with a smooth stable minimal surface inside $U$.  By construction, $\hat{V}=V$ in $M\setminus\overline{U}$.  Also by \eqref{almostminimizingprop}, $\mathcal{H}^2(\hat{V})\geq \mathcal{H}^2(V)$.  But by construction, $$\mathcal{H}^2(\hat{V})\leq\lim_{j\rightarrow\infty} \mathcal{H}^2(\hat{V_j})\leq\lim_{j\rightarrow\infty} \mathcal{H}^2(\Sigma_j) = \mathcal{H}^2(V)$$\noindent Thus $||V||=||V'||$.  The varifold $\hat{V}$ is of course stationary (and stable) in $U$ but it turns out that it is stationary in all of $M$ (see Proposition 7.5 in \cite{cd}). Thus $\hat{V}$ is a replacement for $V$ in $U$ in the sense of Definition \ref{replacement}. 
\begin{rmk}\label{several}
If $U$ is a union of several disjoint domains $\{U_i\}_{i=1}^k$, and $\Sigma$ if a surface in a $3$-manifold intersecting each $U_i$, a solution to \\Problem$(U, Is_j(\Sigma, U))$ coincides in each $U_i$ with a solution to \\Problem$(U_i, Is_j(\Sigma, U_i))$.   This follows because one can concatenate the (disjointly supported) minimizing sequence of isotopies for each of the separate $U_i$ and far out in the minimizing sequence for each $U_i$, the final areas of $\Sigma$ are decreased. 
\end{rmk}
In the final step of the regularity theory, one shows that a varifold with ``enough replacements" is smooth (Proposition 6.3 of \cite{cd}), and we can thus conclude that $V$ is smooth.   Crucial for the following is that once we have proved the smoothness of $V$, it follows by unique continuation for minimal surfaces that $\hat{V}= V$, as long as $U$ is a small enough convex ball.

The upshot of this discussion is that although the convergence of $\Sigma_j$ to $V$ is very weak (i.e. in the sense of varifolds), the convergence of $V_j$ to $V$ is quite strong (i.e., smoothly on compact subsets):  locally, in a small ball $U$ centered in $V$, we can adjust the sequence $\Sigma_j$ to obtain better convergence.  We summarize this discussion in the following lemma that will be used in the proof of Theorem \ref{main}.

\begin{lemma}\label{replacementtheory}
Given a convex set $U\subset\bar{U}\subset B_x(r(x))\setminus\{x\}$, if $V_j$ is a $1/j$ replacement for $\Sigma_j$ in $U$, then $V_j\rightarrow V$, and the convergence is smooth in compact subsets of $U$.
\end{lemma}

Returning to Example 1, we had a sequence of $\Sigma_j=\partial R_{1/j}$ of spheres converging to a torus with multiplicity $2$.  We want to show that such a sequence cannot be almost minimizing and thus cannot arise.   Indeed, if one considers a ball $U$ centered about $(0,1/2,1/2)\in\mathbb{T}^3$ one sees the sequence $\Sigma_j\cap U$ ``folding over"  along a line to converge to $T\cap U$ with multiplicity $2$.  By elementary calculations one can see that the  isotopy that pushes the surfaces to the boundary of $U$ clears out almost all of the area of $\Sigma_j$ without passing through large slices.  Therefore such a sequence cannot be almost minimizing and thus this type of convergence should not be permitted for min-max sequences in the Simon-Smith theory.   In general, even though $\Sigma_j$ are smooth surfaces, their convergence to their limit is so weak as to make it difficult to construct such an isotopy directly.  We will need instead to use the theory of replacements to pass to stable surfaces over which one has more control by Schoen's estimates.  After many reductions, the main step in our argument is still to rule out ``folding along a line," in what we call the No Folding Proposition \ref{nofolding}.

\section{Proof of Lifting Lemma: Proposition \ref{improvedlifting}}\label{improvedlifting2}
In the section we prove the Improved lifting lemma (Proposition  \ref{improvedlifting}) which is the main technical tool in the proof of Theorem \ref{main}.  Let us first explain the main ideas.
\subsection{Strategy}\label{strategy}
The original lifting lemma was proved by covering the closed curve $\gamma\subset\Gamma$ by small disjoint balls such that each consecutive two are contained in a larger ball where one still has the almost-minimizing property for the approximating surfaces $\Sigma_j$.  Consider just the first two such balls in the chain, $B_1,B_2$ contained in $B$ where the min-max sequence is almost minimizing.  The bulk of the argument is then to show that each ``large component" of $\Sigma_j$ contained in $B_1$ connects to some ``large" component of $\Sigma_j$ in $B_2$ in the larger ball $B$.  Thus one can lift $\gamma$ from ball to ball and go around $\gamma$ as many times as it takes for the pigeonhole principle to guarantee that one has landed in the same connected component one began to close up the curve.  This is indirectly ruling out folding as in Example \ref{bad}. 

In order to prove the Improved Lifting Lemma (Proposition \ref{improvedlifting}) one must be able to lift $\gamma$ with the correct multiplicity:  If $\gamma\subset\Gamma$ and $\Gamma$ occurs with multiplicity $n$, one expects to be able to lift $\gamma$ to $n$ disjoint curves in the approximating surfaces, $\Sigma_j$.     We therefore must allow the balls we take along $\gamma$ to \emph{intersect} so that all behavior of $\Sigma_j$ is accounted for, and the structure of $n$ levels or normals graphs persists along all of $\gamma$.  Allowing the balls to intersect, of course, introduces new difficulties.   

We explain the main idea to deal with intersecting balls in the simplified setting of two intersecting balls $B_1$ and $B_2$  along the curve $\gamma$.  Suppose further that $B_1\cup B_2\subset B$ where the sequence $\Sigma_j$ is still almost-minimizing in $B$.  We can first take the $1/j$ replacement $V_j$ for $\Sigma_j$ in $B_1$.  In the interior of $B_1$ , $V_j$ is a union of $n$ normal graphs by Lemma \ref{replacementtheory} (though we have no information regarding how it looks at the boundary of $B_1$).  Then we can adjust $V_j$ so it is smooth over the boundary of $B_1$ and argue that $V_j$ is still almost minimizing in $B$ and therefore also $B_2$.  Thus $V_j$ converges to the same limit as $\Sigma_j$.  Then we can take the $1/j$ replacement of $V_j$ in $B_2$ to arrive at a new surface $W_j$.  Again $W_j$ has the same limit as $\Sigma_j$.  By Schoen's estimates, we obtain that $W_j$ is an honest union of $n$-graphs in $B_2\cup B_1$ except for a small wedge region $W:=(B_2\setminus B'_2)\cap B_1$, where $B'_2$ denotes the ball $B_2$ shrunk slightly.  Bridging this gulf comprises what we call the No Folding Proposition \ref{nofolding} - but now $W_j$ restricted to $W$ is a minimal surface, no longer an arbitrary smooth varifold not satisfying any PDE.  We then use an integrated Gauss-Bonnet argument to bound the $L^2$ norm of the second fundamental form of $W_j$ over the wedge region $W$.   This implies graphical smooth convergence of $W_j$ away from finitely many points, so by perturbing the curve $\gamma$ slightly to avoid these points, we obtain the desired result.

Our argument is fundamentally different from the original lifting lemma in that we will at the outset abondon all hope to lift any curves to the original min-max sequence.  Instead we use stable replacements of the sequence and lift curves there to take advantage of Schoen's curvature estimates. 
\\
\\
\noindent
\emph{Proof of Improved Lifting Proposition \ref{improvedlifting}:}
\subsection{Set up}
Let  $\Sigma_j$ be a $1/j$-almost minimizing min-max sequence of genus $g$ surfaces.  Without loss of generality, we can assume that $\Sigma_j\rightarrow n\Gamma$, where $\Gamma$ is a connected smooth embedded minimal surface and $n$ is a positive integer.  By Theorem \ref{am}, we obtain that for each $x\in\Gamma$, there exists an $r(x) > 0$ such that $\Sigma_j$ is $1/j$ almost minimizing in any annulus centered at $x$ of outer radius at most $r(x)$ when $j$ is large enough.  Consider the covering of $\Gamma$ given by $\cup_{x\in\Gamma}B_x(r(x))$.  By compactness of $\Gamma$, there is a finite set $\mathcal{S}=\left\{x_1,x_2,...,x_n\right\}$ such that the balls $\left\{B_{x_i}(r(x_i))\right\}_{i=1}^n$ cover $\Gamma$.  Thus for any $x\in\Gamma\setminus\left\{x_1,x_2,...,x_n\right\}$, there is a radius $\rho(x)$ such that $\Sigma_j$ is $1/j$ almost minimizing in $B_x(\rho(x))$ for $j$ large enough: for $x\in B_{x_k}(r(x_k))\setminus\mathcal{S}$, set 
\begin{equation}
\rho(x) = \frac{1}{2}\text{min}(\text{dist}(x,\partial B_{x_k}(r(x_k))),\text{dist}(x,x_k)).
\end{equation}

By Lemma \ref{genuscollapse}, after passing to a subsequence of $\Sigma_j$ (not relabeled) there is another finite set of points $G=\left\{y_1,y_2,...,y_k\right\}$ such that if $x\in\Gamma\setminus G$, there is a radius $r'(x)$ such that $g(\Sigma_j\cap B_x(r(x))=0$ for $r(x)\leq r'(x)$.  For $x\in\Gamma\setminus\left\{x_1,...x_n,y_1,...,y_k\right\}$, set $r''(x)=\min(r(x),r'(x))$ and define $\mathcal{B}:=\left\{x_1,...x_n,y_1,...,y_k\right\}$.  Then for any $x\in\Gamma\setminus\mathcal{B}$ there is a radius $r''(x):=\min(\rho(x),r'(x))$ such that 
\begin{enumerate}
\item $g(\Sigma_j\cap B_x(r(x)))=0$ for $j$ large enough and $r(x)\leq r''(x)$,
\item $\Sigma_j$ is $1/j$-almost minimizing in $B_x(r(x))$ for $j$ large enough and  $r(x)\leq r''(x)$.
\end{enumerate}

Now we proceed to proving Proposition \ref{improvedlifting}.  For simplicity we will first assume $k=1$ (i.e, that we are lifting one curve $\gamma:=\gamma_1$) parameterized by arc-length.  At the end of the proof (Step 9 in the outline), we explain the straightforward changes needed to lift multiple curves.  Let us choose $\epsilon_0$ smaller than the injectivity radius of $\Gamma$, the convexity radius of $M$, the scale on which the monotonicity formula \ref{mono} holds, the radius in Appendix A, and the small $\eta$ for which the projection map from $T_\eta(\Gamma)$ to $\Gamma$ is well-defined.  Fix then $\epsilon<\epsilon_0$.

If $\gamma$ intersects the set $\mathcal{B}$, perturb it slightly to avoid the finite set and so that it still lies in $T_{\epsilon}(\gamma)$ and is homotopic to $\gamma$ (we still call the perturbed curve $\gamma$).  By the preceding discussion, for any $x\in\gamma$, there exists a ball centered at $x$ in which $\Sigma_j$ contains no genus and is $1/j$-almost minimizing for $j$ large enough.

For any small $r<\epsilon$, consider $2k$ points $p_0,p_1, ..., p_{2k-1}$ spread along $\gamma$ whose distance in $\Gamma$ between neighboring points is $r$ (by convention set $p_{2k}:=p_0$).  After small homotopy supported within $T_\epsilon(\gamma)$, by shrinking $r$, we can further suppose that for each $i\in\{0,...,2k-1\}$, the curve $\gamma(s)$ restricted to the interval $[p_i,p_{i+1}]$ coincides with the geodesic segment of length $r$ in $\Gamma$ joining $p_i$ to $p_{i+1}$.   The family of ``small" balls $\cup_{i=0}^{2k-1}B_{p_i}(3r/4)$ is a covering of $\gamma$.  We also consider the covering of $\gamma$ by a family of $k$ ``large" balls $\cup_{i=0}^{k-1}B_{p_{2i+1}}(15r/8)$ (based only around the points $p_i$ when $i$ is odd).  By choosing $r$ small enough, these large balls are all contained still in $T_\epsilon(\gamma)$.  Setting $r_1:=3r/4$ and $r_2:=15r/8$ we choose $r$ and the location of the balls so that in addition:
\begin{enumerate}
\item For $j\in\{0,\ldots,2k-1\}$, $B_{p_j}(r_1)$ intersects $B_{p_{j-1}}(r_1)$ and $B_{p_{j+1}}(r_1)$ (here if $p_j = p_0$ we set $p_{j-1}:= p_{2k-1}$ and if $p_j = p_{2k-1}$ we set $p_{j+1} := p_0$) and is disjoint from the other balls $B_{p_l}(r_1)$ along $\gamma$,
\item For $j\in\{0,1,..k-1\}$ we have $\bar{B}_{p_{2j}}(r_1)\cup\bar{B}_{p_{2j+1}}(r_1)\cup\bar{B}_{p_{2j+2}}(r_1)\subset B_{p_{2j+1}}(r_2)$,
\item For $j\in\{0,1,..k-1\}$, if $r$ is even and $B_{p_{r}}(r_1)$ is not disjoint from $B_{p_{2j+1}}(r_2)$, then $r=2j$ or $r=2j+2$,
\item $\Sigma_j$ is $1/j$ almost minimizing in $B_{p_{2i+1}}(r_2)$ for $i\in\{0,1,\ldots k-1\}$ and $j$ large enough,
\item $\Gamma$ intersects the boundaries of the balls  $B_0(r),B_1(r)\ldots,B_{2k-1}(r)$ transversally for all $r$ in a neighborhood of $r_1$
\item  $g(\Sigma_j\cap B_{p_{2i+1}}(r_2)) =0$ for  $i\in\{0,1,\ldots k-1\}$.
\end{enumerate}
\noindent
Loosely speaking, each large ball of radius $r_2$ contains three of the consecutive smaller balls.  Each small ball of radius $r_1$ based about $p_i$ is contained entirely in two of the larger balls when $i$ is even, while if $i$ is odd it is contained entirely in only one of the larger balls (and intersects two others).  The choice of $r_1$ and $r_2$ is simply to guarantee the good intersecting properties (1), (2) and (3) in Euclidean space, which then also hold on a small enough scale in a Riemannian manifold.  Item (5) can also be guaranteed because at a small enough scale, $\Gamma$ looks roughly flat.

Let us denote  $B_{p_i}(r_1)$ by $B_i$ and the ball $B_{p_i}(r_1(1-\delta))$ by $B_i(1-\delta)$.  Also set $B_E:=\bigcup_{i\in\{0,2,4...2k-2\}}B_i$ and  $B_O := \bigcup_{i\in\{1,3,\ldots,2k-1\}} B_i$.   

Moreover, again since $\Gamma$ looks roughly flat on a small enough scale, by further shrinking of $r$, we can arrange that for all $i$,
\begin{equation}\label{nb}
\exp_{p_i}^{-1}(\gamma\cap B_i)\subset T_{r/100}(\{\mbox{some line through the origin in } T_{p_i}M\})
\end{equation}
In other words, $\gamma$ roughly speaking cuts the balls comprisng $B_E\cup B_O$ in half.
\subsection{Outline of Argument}
The remainder of the proof of Proposition \ref{improvedlifting} consists of nine steps, each of which we elaborate on afterwards as necessary:
\\
\\
\textbf{Step 1:}   By (4) in the Setup, $\Sigma_j$ is $1/j$ almost minimizing in each of the even balls comprising $B_E$.
Let $V_j$ be the limit of a minimizing sequence to Problem($\Sigma_{j}$,$Is_{j}(\Sigma_j,B_E)$). Recall that by Remark \ref{several}, we can choose the minimizing sequence so that $V_j$ coincides with the $1/j$-replacement of $\Sigma_j$ in each of the balls comprising $B_E$. By Lemma \ref{replacementtheory}, $V_j$ is a stable minimal surface (with multiplicity $1$) inside $B_E$ converging to $n\Gamma$ smoothly on compact subsets of $B_E$.  Fix $\delta>0$ small enough so that $r_1(1-\delta)$ is contained within the interval provided by (5) for each $B_i$.  The surface $V_j$ restricted to $B_0(1-\delta)$ is a union of $n$ normal graphs over $\Gamma$. Therefore $V_j\cap\partial B_1\cap B_0(1-\delta)$ is a union of $n$ segments  $\left\{\alpha_i\right\}_{i=1}^n$.  Without loss of generality we will henceforth only focus on lifting the curve $\gamma$ from the center of $B_0$ to the center of $B_1$.
\\
\\
\textbf{Step 2:}
We prove that $V_j$ from Step 1 arises topologically from $\Sigma_j$ after finitely many surgeries.  In particular $V_j$ contains no genus in any of the balls comprising $B_E\cup B_O$.  This step is achieved in Section \ref{aftersur}.
\\
\\
\textbf{Step 3:}
Since $\alpha_i\subset B_0(1-\delta)$ and thus dist($\alpha_i,\partial B_0)>\delta r_1$, we obtain by Schoen's estimates \eqref{schoen} that $|A_{V_j}|^2$ restricted to the curves $\alpha_i$ is bounded by a fixed constant depending only on $\delta r_1$. By Lemma \ref{intersection} we can then bound the ambient curvature of the curves $\alpha_i:$
\begin{equation}
|k_M(\alpha_i)|\leq\frac{|A_{V_j}|(\alpha_i)+|A_{\partial B_1}|(\alpha_i)}{|\sin\alpha|},
\end{equation}
where $\alpha$ is the angle between $V_j$ and $\partial B_1$.  But by construction (item (5) in the Setup) the angle between $\Gamma$ and $\partial B_1$ is bounded away from 0 so that by the smooth convergence of $V_j$ to $\Gamma$ in $B_0(1-\delta)$ the angle between $V_j$ and $\partial B_1\cap B_0(1-\delta)$ is also bounded away from zero when $j$ is large.  Thus altogether we obtain bounds (independent of $j$) on the ambient curvature of the $n$ curves $\{\alpha_i\}_{i=1}^n$ comprising $V_j\cap\partial B_1\cap B_0(1-\delta)$.  In other words, for some $C>0$ independent of $j$, 
\begin{equation}\label{ambient1}
|k_M(\alpha_i)|\leq C\mbox{  for  }1\leq i\leq n.
\end{equation}
\\
\textbf{Step 4:} $V_j$ will not be smooth at the boundary of $B_E$, so one must smooth the sequence $V_j$ slightly outside of $B_E$ to produce $\tilde{V}_j$ (so that the area of the symmetric difference of $V_j$ and $\tilde{V}_j$ approaches $0$ as $j\rightarrow\infty$ and $\tilde{V}_j\rightarrow\Gamma$.)  We then argue that $\tilde{V}_j$ is still $1/j$-almost minimizing in the larger balls $B_{p_1}(r_2),...,B_{2k-1}(r_2)$ and therefore also in each of the odd balls comprising $B_O$ since each odd ball is contained in some ``large" ball by item (2) in the Setup. (Recall that if $\Sigma$ is $1/j$-almost minimizing in $U$, then $\Sigma$ is $1/j$-almost minimizing in any subset of $U$.)
\\
\\
\textbf{Step 5:}   We can now consider a minimizing sequence for Problem($\tilde{V}_j$, $Is_j( \tilde{V}_j,B_O)$), and suppose this minimizing sequence converges to $W_j$.  As in Step 1 by Lemma \ref{replacementtheory}, the sequence $W_j$ converges to $n\Gamma$ and restricted to $B_1(1-\delta)$ consists of a union of $n$ graphs for $j$ large enough.  As in Step 3 we have uniform bounds on the ambient curvature of the $n$ curves $\left\{\beta_i\right\}_1^n$ comprising $W_j\cap B_0(1-\delta)\cap\partial B_1(1-\delta)$ (again we focus on only the first jump, from ball $B_0$ to $B_1$).  \\
\\
We summarize our progress so far:
\begin{enumerate}
\item $W_j$ converges smoothly as $n$ graphs to $\Gamma$ on $B_0(1-\delta)\setminus B_1$ 
\item $W_j$ converges smoothly as $n$ graphs to $\Gamma$ on $B_1(1-\delta)$
\item $|k_M(\alpha_i)|\leq C\mbox {    for    }1\leq i\leq n$
\item $|k_M(\beta_i)|\leq C\mbox {    for    }1\leq i\leq n$
\end{enumerate}
\noindent
\textbf{Step 6:}   It remains to show that $W_j$ restricted to the tiny wedge region $$W = B_0(1-\delta)\cap (B_1\setminus B_1(1-\delta))$$ also consists of $n$ normal graphs.  The wedge $W$ is a three-ball with piecewise smooth boundary. Unfortunately we will not be able to prove that $W_j$ consists of $n$ graphs in $W$ since we cannot rule out small necks being pushed to the boundary of $W$.  We use instead a Gauss-Bonnet argument using (3) and (4) of Step 5 and comprising Sections \ref{beg} and \ref{beg2}.  Specifically we find a smaller region $\tilde{W}\subset W$ containing a neighborhood of $\gamma\cap W$ that connects $\partial B_1\cap B_0(1-\delta)$ to $\partial B_1(1-\delta)\cap B_0(1-\delta)$ and on which one has the bound:
\begin{equation}\label{ye}
\sup_j\int_{\tilde{W}\cap W_j}|A_{W_j}|^2 <\infty.
\end{equation}
\noindent
\textbf{Step 7:}  By standard results due to Choi-Schoen \cite{cs} (and an extension by Brian White \cite{W}) \eqref{ye} implies that away from finitely many points in $\tilde{W}$, $W_j$ converges as graphs to $\Gamma$.  Since the singular set consists of finitely many points, we can lift our original curve $\gamma$ through $\tilde{W}$ on $n$ different possible levels avoiding this set.  In other words, adjusting $\gamma$ slightly again $T_\delta(\gamma)$ is a union of $n$ normal graphs for $\delta$ small enough.  We give more details in Section \ref{wh}.
\\
\\
\noindent
\textbf{Step 8:}  Together Steps 1-7 allow us to perturb $\gamma$ within $T_\epsilon(\gamma)$ and perform surgeries on $\Sigma_j$ to obtain $\tilde{\Sigma}_j$ so that in some neighborhood of $\gamma$, $\tilde{\Sigma}_j$ consists of $n$ normal graphs converging to $n\Gamma$.   Thus if $\Gamma$ is orientable, it follows that $p^{-1}(\gamma)\cap T_\epsilon(\Gamma)\cap\tilde{\Sigma}_j$ consists of $n$ closed curves each projecting with degree $1$ onto $\gamma$.  Similarly, if $\Gamma$ is non-orientable, and $p^{-1}(\gamma_i)\cap T_\epsilon(\Gamma)$ is a M\"obius band, then $p^{-1}(\gamma)\cap T_\epsilon(\Gamma)\cap\tilde{\Sigma}_j$ consists of $n/2$ connected curves, each double covering $\gamma$ via nearest point projection.  The integer $n$ must be even in this case: otherwise $p^{-1}(\gamma)\cap T_\epsilon(\Gamma)\cap\tilde{\Sigma}_j$ contains a curve isotopic to the core of the M\"obius band $p^{-1}(\gamma)\cap T_\epsilon(\Gamma)$.  The normal bundle over this curve is non-trivial, and thus $\tilde{\Sigma}_j$ would also be non-orientable, contradicting the fact that $\tilde{\Sigma}_j$ arises from surgeries performed on an orientable surface. 
\\
\\
\noindent
\textbf{Step 9:} Lastly we explain in Section \ref{multiple} the changes needed to lift multiple curves $\{\gamma_i\}_{i=1}^k\subset\Gamma$. 

\subsection{Completion of Step 4: ``$V_j$ still almost minimizing"}\label{stillam}
Here we prove that the $V_j$ constructed in the outline can be smoothed slightly at the boundary of $B_E$ to give a new sequence $\tilde{V}_j$ which has the same limit in $B_E\cup B_O$ and is $1/j$-almost minimizing in each of the balls comprising $B_O$.   This statement follows from the following general lemma:
\begin{lemma}\label{stillam2}
Let $B$ an be open ball in a 3-manifold of sufficiently small radius, and $A$ a union of balls contained in $B$ so that $\overline{A}\subset B$.  Let $\Sigma$ be a smooth embedded surface that is $1/j$-almost minimizing in $B$ with $\partial\Sigma\subset\partial B$.   Let $\phi^l(1,\Sigma)$ be a minimizing sequence to Problem($\Sigma$, $Is_{j}(\Sigma,A)$) such that $\phi^l(1,\Sigma)\rightarrow V$.  Then there exists $\epsilon_1>0$ so that for $\epsilon<\epsilon_1$, there exists a smooth surface $\hat{V}$ such that
\begin{enumerate}
\item $\hat{V}=V$ in $B\setminus (T_\epsilon(A)\setminus A)$ 
\item $\hat{V}$  is $1/j$-almost minimizing in $B$ 
\item  $|\mathcal{H}^2(\hat{V})-\mathcal{H}^2(V)|\leq\epsilon$.
\end{enumerate}
\end{lemma}
\begin{rmk}
For an open set $A$, $T_\epsilon(A)$ denotes all points within distance $\epsilon$ of $A$.
\end{rmk}
 \begin{proof}
We can assume without loss of generality that $\mathcal{H}^2(V)<\mathcal{H}^2(\Sigma)$.  Fix $\epsilon_1:=(\mathcal{H}^2(\Sigma)-\mathcal{H}^2(V))/2\leq 1/2j$.
By Lemma 7.4 in \cite{cd}, $V$ is a smooth minimal surface in $A$ which by Lemma 8.1 in \cite{dp} occurs with multiplicity $1$ and has boundary coinciding with that of $\Sigma$ in $\partial A$.  The surface $V$, however, might not be smooth over $\partial A$.  To correct this, for any fixed $\epsilon<\epsilon_1$ take a sequence of isotopies $\phi^l_2$ supported in $T_\epsilon(A)\setminus\overline{A}$ such that
\begin{enumerate}
\setlength{\itemsep}{1pt}
  \setlength{\parskip}{0pt}
  \setlength{\parsep}{0pt}

\item $|\mathcal{H}^2(\phi^l_2(t,\Sigma))-\mathcal{H}^2(\Sigma)|\leq\epsilon$ for all $0\leq t\leq 1$
\item $(\phi^l_2\circ\phi^l)(1,\Sigma)\rightarrow\hat{V}$ where $\hat{V}$ is smooth in $B$
\item $\hat{V}=V$ in $A$.
\end{enumerate}
\noindent
Note from the definition of $\phi^l$ we obtain
\begin{equation}\label{am13}
\mathcal{H}^2(\phi^l(t,\Sigma))\leq \mathcal{H}^2(\Sigma)+\frac{1}{8j}.
\end{equation}   
\noindent
We also obtain for large $l$, 
\begin{align}
\mathcal{H}^2(\phi_2^l(t,\phi^l(1,\Sigma))) & \leq\mathcal{H}^2(\phi^l(1,\Sigma))+\epsilon  \\&\leq \mathcal{H}^2(V)+2\epsilon \\&\leq \mathcal{H}^2(V)+(\mathcal{H}^2(\Sigma)-\mathcal{H}^2(V))\label{de}\\& \leq \mathcal{H}^2(\Sigma) \label{f}.
\end{align}
In \eqref{de} we have used that $\epsilon<\epsilon_1$.  Together \eqref{f} and \eqref{am13} imply that $\phi_2^l\circ\phi^l$ is contained in $Is_{j}(\Sigma,B)$. 

Suppose toward a contradiction that $\hat{V}$ were not $1/j$-almost minimizing in $B$.  Then there exists an isotopy $\psi$ supported in $B$ such that
\begin{equation}\label{am11}
 \mathcal{H}^2(\psi(1,\hat{V}))\leq \mathcal{H}^2(\hat{V})-1/j,
\end{equation}
and
\begin{equation}\label{am12}
 \mathcal{H}^2(\psi(t,\hat{V}))\leq \mathcal{H}^2(\hat{V})+1/{8j}\mbox{ for all }0\leq t\leq 1.
\end{equation}

We will now show that \eqref{am11} and \eqref{am12} contradict the fact that $\Sigma$ is $1/j$-almost minimizing in $B$.

Let $\phi^l_3$ denote the concatenated isotopy $\phi^l_2\circ\phi^l$.  Since $\phi^l_3(1,\Sigma)\rightarrow\hat{V}$ in the sense of varifolds, we obtain by continuity that for $l$ large enough,  
\begin{equation}\label{c}
\mathcal{H}^2(\psi(t,\phi^l_3(1,\Sigma)))-\mathcal{H}^2(\psi(t,\hat{V})))\leq\epsilon\mbox{    for all   } 0\leq t\leq 1.
\end{equation}
\noindent
Thus by \eqref{c}, \eqref{am11} and our choice of $\epsilon$, we obtain
\begin{align}
\mathcal{H}^2(\psi(1,\phi_3^l(1,\Sigma))) &\leq \mathcal{H}^2(\psi(1,\hat{V}))+\epsilon\\&\leq \mathcal{H}^2(\hat{V})+\epsilon-1/j\\&\leq\mathcal{H}^2(V)+2\epsilon-1/j\\&\leq\mathcal{H}^2(\Sigma)-1/j\label{fff}.
\end{align}
By \eqref{am12}, \eqref{c}, and our choice of $\epsilon$ we obtain
\begin{align}
\mathcal{H}^2(\psi(t,\phi_3^l(1,\Sigma)))& \leq \mathcal{H}^2(\hat{V})+\frac{1}{8j}+\epsilon\\&\leq \mathcal{H}^2(V)+\frac{1}{8j}+2\epsilon \\&\leq\mathcal{H}^2(\Sigma)+\frac{1}{8j}\label{ff}.
\end{align}
\noindent
Equations \eqref{am13}, \eqref{f} and \eqref{ff} imply that for $l$ large, the isotopy $\psi\circ\phi_2^l\circ\phi^l$ is contained in $Is_{j}(\Sigma, B)$.  Thus \eqref{fff} contradicts the fact that $\Sigma$ is $1/j$-almost minimizing in $B$.
\end{proof}

Given Lemma \ref{stillam2}, we can easily complete Step 4.  Apply Lemma \ref{stillam2} successively to $\Sigma_j$ with  $B:=B_{p_{2i+1}}(r_2)$ and $A:=B_{2i}\cup B_{2i+2}$ for each $i\in\{0,1,...,k-1\}$ (by item (3) in the Setup, there are no other even balls aside from $A$ intersecting each $B$).

\subsection{Completion of Step 6: Crossing the gap}\label{beg}
The goal of this section is to begin a proof of Step 6 in the outline.  Step 6 will be completed in the next subsection. 
We recall the Gauss-Bonnet formula with boundary for our sequence of minimal surface $W_j$ intersected transversally with an open set $B\subset B_E\cup B_O$.  We will assume that $B$ is diffeomorphic to a three-ball but has piecewise smooth boundary.

\begin{align}\label{b}
\int_{W_j\cap B} K_{W_j} + \int_{\partial B\cap W_j}k_g =& 2\pi(2n(W_j\cap B)-2g(W_j\cap B))\nonumber\\
									&  -2\pi e(W_j\cap B)-T(W_j\cap\partial B)
\end{align}

Here $e(W_j\cap B)$ is the number of components of $W_j$ in $\partial B$, $n(W_j\cap B)$ is the number of components of $W_j$ in $B$, $g(W_j\cap B)$ is the genus of $W_j$ in $B$ and $T(W_j \cap\partial B)$ denotes the total jump angle of $W_j \cap\partial B$ at points where the boundary of $\partial W_j$ crosses the non-smooth part of $\partial B$.  Recall that if $n(W_j\cap B)>1$, then $g(W_j\cap B)$ is defined to be the sum of the genera of each component.  
 By minimality and the Gauss equation, we obtain 
\begin{equation}\label{gauss}
K_{W_j} = \text{sec}_M(e_1,e_2) - \frac{|A|^2}{2},
\end{equation}
where $\text{sec}_M(e_1,e_2)$ is the sectional curvature of $M$ in the plane determined by an orthonormal frame of $W_j$.  Thus we obtain by plugging \eqref{gauss} into \eqref{b}:

\begin{align}\label{w1}
\int_{W_j\cap B} |A|^2 = &  2\int_{W_j\cap B} \text{sec}_M(e_1,e_2) + 2\int_{\partial B\cap W_j}k_g\nonumber \\
                                         &+8\pi g(W_j\cap B) -8\pi n(W_j\cap B) +4\pi e(W_j\cap B)\nonumber\\&
			              +2T(W_j\cap\partial B)
\end{align}
Discarding the $n(W_j\cap B)$ term since it has a favorable sign, and taking absolute value we obtain from \eqref{w1}:

\begin{align}\label{nice}
\int_{W_j\cap B} |A|^2\leq &  2 \mathcal{H}^2(W_j\cap B))\sup_M|\text{sec}_M|+ 2\int_{\partial B\cap W_j}|k_g|\nonumber\\
                                              & +8\pi g(W_i\cap B) + 4\pi e(W_j\cap B)+2|T(W_j\cap\partial B)|
\end{align}
Note that $g(W_j\cap B) \leq g$.  Also $$\sup_j\mathcal{H}^2(W_j\cap B)<\infty,$$ since $W_j$ converges to $ n\Gamma$ as varifolds.  

For a curve $\gamma(s)$ parameterized by arclength that is contained in the submanifold $W_j$, the norm of the vector $\nabla^M_{\dot{\gamma}(s)}\dot{\gamma}(s)$ denotes the ambient curvature $|k_M|(\dot{\gamma}(s))$ of $\dot{\gamma}(s)$ while the geodesic curvature $k_g(\gamma(s))$ of $\gamma(s)$ considered a curve in $W_j$ is the norm of the vector $\nabla^{W_j}_{\dot{\gamma}(s)}\dot{\gamma}(s)$, i.e., $\nabla^{M}_{\dot{\gamma}(s)}\dot{\gamma}(s)$ projected onto $W_j$.  Thus we obtain, $|k_g|\leq|k_M|$ along any such curve.   Combining this with \eqref{nice} we obtain:
\begin{equation}\label{ifonly}
\int_{W_j\cap B} |A|^2\leq C +  2\int_{\partial B\cap W_j}|k_M|+ 4\pi e(W_j\cap B)  +2|T(W_j\cap\partial B)|,
\end{equation}
\noindent
where $C$ does not depend on $j$. 
We would like in \eqref{ifonly} to set $B = W$ (recall $W$ is the wedge region $B_1\setminus B_1(1-\delta)\cap B_0(1-\delta)$).  Thus by \eqref{ifonly}  if we could bound the curvature $|k_M|$ of $\partial B\cap W_j$ as well as $e(W_j\cap B)$ and the total jump angle independently of $j$, we would obtain

\begin{equation}\label{l2bound}
\sup_j\int_{W_j\cap W}|A|^2<\infty,
\end{equation}
as desired.
\indent

Unfortunately one has no control on the curvature of the boundary $W\cap W_j$ or of the number of boundary components.  Instead we use a trick of Ilmanen (cf.\ Lemma 1 in \cite{i}) to \emph{average} \eqref{ifonly} over slight shrinkings of $W$ in order to obtain an $L^2$ bound for $|A|$ on a slightly smaller region than $W$ that nonetheless connects $\partial B_1\cap B_0(1-\delta)$ to $\partial B_1(1-\delta)\cap B_0(1-\delta)$ and contains $\gamma$. \\
\indent
 In order to formulate the result, we first parametrize the region $W$.  Consider exponential normal coordinates around $p = p_1$. $$\exp_p:B_{r_1}(0)\subset T_pM\rightarrow M$$  By \eqref{nb}, after rotation in $T_pM$ we can guarantee that 
\begin{equation}\label{nbd}
\exp_p^{-1}(\gamma\cap B_1)\subset T_{\eta}(\{x\mbox{-axis in }T_p M\})
\end{equation}\noindent for some small $\eta<< r$.   We then use spherical coordinates $(\rho,\theta,\phi)$ to parameterize the vector space $T_p M$, where $\theta$ is the polar angle (normalized to be zero on the negative $x$-axis) and $\phi$ is the azuthimal-coordinate.  For any $\phi_0<\pi/2$ and $\theta_0<2\pi$ we define the wedge region:
\begin{equation} 
R_{\theta_0,\phi_0} = \exp_p(\{(\rho,\theta,\phi)\in T_pM: (1-\delta)r_1\leq\rho\leq r_1, |\phi-\pi/2|\leq\phi_0, |\theta|\leq\theta_0\}).
\end{equation}
\noindent
By \eqref{nbd} we can choose $\theta_0$ and $\phi_0$ suitably small so that
\begin{enumerate}
\item $\gamma\cap W\subset R_{\theta_0/2,\phi_0}$.  
\item For $j$ large, $W_j\cap\partial  R_{\theta_0,\phi_0}\cap\{\rho=r_1\}$ and $W_j\cap\partial  R_{\theta_0,\phi_0}\cap\{\rho=r_1(1-\delta)\}$ each consist of $n$ curves with bounded curvature.
\end{enumerate}
\noindent
Item (2) follows from Steps 1) and 5) and item (1) follows from \eqref{nbd}.  Moreover since $\Gamma$ is flat on a small enough scale, by shrinking $r$ if necessary, we can further guarantee
\begin{equation}\label{surfacecontained}
\Gamma\cap R_{\theta_0,\phi_0} \subset R_{\theta_0,\phi_0/2}.
\end{equation}
\noindent
We prove the following
\begin{prop} [No Folding]\label{nofolding}  
\begin{equation}\label{whatwewant}
\sup_j \int_{W_j\cap R_{\theta_0/2,\phi_0}}|A|^2 < \infty.
\end{equation}
\end{prop}
\noindent 
Proposition \ref{nofolding} rules out behavior as in Example \ref{bad}. 
\begin{rmk} As a heuristic justification for why Proposition \ref{nofolding} rules out folding,  consider the surfaces $C_r\times [0,1]$ in $\mathbb{R}^3$, where $C_r$ is a semi-circle of radius $r$.  Thus the ``folding" occurs along a line of length one.  The principal curvatures are $1/r$ and $0$. Then $|A|^2 =1/r^2$ and $\mathcal{H}^2(C_r\times [0,1])=\pi r$.  Thus $\int_{C_r\times[0,1]} |A|^2 = \pi/r$ and $\int_{C_r\times[0,1]} |A|^2\rightarrow\infty$ as $r\rightarrow 0$.  
\end{rmk}
\begin{rmk}
It is well-known that in a fixed 3-manifold, an area and genus bound for closed minimal surfaces imply a bound on the $L^2$ norm of the second fundamental form (by using the Gauss-Bonnet formula, the Gauss equation, and minimality).  Thus Proposition \ref{nofolding} can be interpreted as a local version of this fact. One could use the arguments here to show that an area and genus bound for minimal surfaces in a fixed ball imply an $L^2$ bound on the second fundamental form in a slightly smaller ball (cf.\ Lemma 1 in \cite{i}).  The technical complication below is that we don't want bounds on a compactly supported interior domain but on a domain part of whose boundary coincides with that of the larger ball.  But the whole point is that using stability and Schoen's estimates we have good curvature bounds for the region where the interior domain touches the exterior domain.
\end{rmk}
\subsection{Completion of Step 6: No Folding}\label{beg2}
The key to proving the No Folding Proposition \ref{nofolding} is that by Step 3 of the outline, the curvature of $W_j$ on ``most" of the boundary of the region $(B_1\setminus B_1(1-\delta))\cap B_0(1-\delta)$ is bounded. 
We now proceed to the proof. Throughout the argument $C$ will denote a constant independent of $j$, changing from line to line and possibly appearing in the same equation with different values. 
\\
\\
\noindent
\emph{Proof of No Folding Proposition \ref{nofolding}}:  
\\
\noindent
\\
\textbf{Step 1: Finding a good cutoff function}\\
We need a cutoff function $\psi(\rho,\theta,\phi)$ defined on $R_{\theta_0,\phi_0}$ in $M$ with the following properties:
\begin{enumerate}
\item $\psi$ restricted to $R_{\theta_0/2,\phi_0}$ is $1$
\item $\psi$ only depends on $\theta$,
\item $\frac{|\nabla^M\psi|^2}{\psi}\leq C$ 
\end{enumerate}
\noindent
We first define a Lipschitz function $f:[-1,1]\rightarrow\mathbb{R}$ as follows:
\begin{equation}
f(x) = \begin{cases} 1 & \mbox{ if } |x|\leq 1/2\\
                                     4(1-x)^2 & \mbox{ if } |x|\geq 1/2\end{cases}
\end{equation}
Notice that we have $f'(x)^2\leq 16f(x)$.  Set $\psi(\rho,\theta,\phi):= f(\theta/\theta_0)$.   The function $\psi$ then satisfies (1), (2), and (3).


Note that our cutoff function is unusual in that it does not have compact support on its domain - it is equal to $1$ on the part of $\partial R_{\theta_0/2,\phi_0}$ where we have control over the curvature of $W_j$.  Since $\psi$ is only defined on $R_{\theta_0,\phi_0}$, all level sets or superlevel sets of $\psi$ are understood to be contained in $R_{\theta_0,\phi_0}$.
\\
\\
\noindent

\noindent
\textbf{Step 2: Using the ambient boundary curvature bounds}\\
For $0\leq t\leq 1$ the region $\left\{\psi > t\right\}\cap R_{\theta_0,\phi_0}$ is a ball with piecewise smooth boundary comprised of six parts.  The first two are the two disks comprising $K_t := \left\{\psi = t\right\}$ where $\theta$ is constant (the ``left" and ``right" components). The third we denote $S_t^1$ is the component contained in $\partial B_1\cap B_0(1-\delta)$.  The fourth component we denote $S_t^2$ is the part of $\partial(\left\{\psi > t\right\})$ contained in $\partial B_1(1-\delta)\cap B_0(1-\delta)$.   The fifth and sixth are the ``top" and ``bottom" components where $\phi$ is equal to either $\pi/2-\phi_0$ or $\pi/2+\phi_0$.  However, by the monotonicity formula \eqref{mono} and \eqref{surfacecontained}, $W_j$ is disjoint from these last two components when $j$ is sufficiently large.

We apply \eqref{ifonly} to $B_t := \left\{\psi > t\right\}\cap R_{\theta_0,\phi_0}$ which yields 

\begin{align}\label{first}
\int_{W_j\cap\left\{\psi>t\right\}}|A|^2\leq & C + 2\int_{\left\{\psi=t\right\}\cap W_j}|k_M|+ 2\int_{S_t^1\cap W_j}|k_M| + 2\int_{S_t^2\cap W_j}|k_M|\\
&+4\pi e(W_j\cap\left\{\psi>t\right\})+2|T(W_j\cap\partial\{\psi>t\})|.
\end{align}

From equations 3) and 4) in Step 5 of the outline and Lemma \ref{intersection} we obtain that the ambient curvature $k_M$ of the curves given by $W_j\cap S^1_t$ and $W_j\cap S^2_t$ is bounded independently of $j$.   Moreover, the lengths of these curves are bounded independently of $j$ because of the smooth convergence $W_j\rightarrow n\Gamma$ restricted to these curves.  Thus we obtain the estimate

\begin{align}
\int_{W_j\cap\left\{\psi>t\right\}}|A|^2\leq & C+ C\int_{\left\{\psi=t\right\}\cap W_j}|k_M|+Ce(W_j\cap\left\{\psi>t\right\})\nonumber\\&+2|T(W_j\cap\partial\{\psi>t\})|.\label{hum}
\end{align}

\noindent
\textbf{Step 3: Handling the number-of-ends term from the Gauss-Bonnet formula}\\ 
\noindent
We now explain how to handle the term $e(W_j\cap\left\{\psi>t\right\})$ for the number of boundary components of $W_j$ in $K_t\cup S_t^1\cup S_t^2$ and also the total turning angle $T(W_j\cap\partial\{\psi>t\})$.  Since $W_j$ restricted to $S_t^1$ and $S_t^2$ consists of $n$ curves, the connected components of these curves number at most $2n$.  Any other curve on $W_j$ in $K_t\cup S_t^1\cup S_t^2$ must be contained entirely in $K_t$.  The set $K_t$ consists of two regions, each of constant $\theta$ (the ``left" and ``right" pieces).    But by Appendix A, any closed curve $\gamma$ contained in $K_t$ satisfies
\begin{equation}
\int_{\gamma} |k_M|ds\geq\pi.
\end{equation} 
Thus for large $j$ we obtain
\begin{equation}\label{e}
e(W_j\cap\{\psi>t\})\leq 2n + \frac{1}{\pi}\int_{\{\psi=t\}\cap W_j} |k_M|ds.
\end{equation}
Moreover, since there are at most $4n$ points where curves on $\partial \{\psi>t\}$ meet the non-smooth parts of the boundary, we obtain
\begin{equation}\label{turn}
|T(W_j\cap\partial\{\psi>t\})|\leq4\pi n.
\end{equation}
 Plugging \eqref{e} and \eqref{turn} into \eqref{hum} we get
\begin{equation}\label{first}
\int_{W_j\cap \left\{\psi>t\right\}}|A|^2\leq  C + C\int_{\left\{\psi=t\right\}\cap W_j}|k_M|ds.
\end{equation}
\\
\noindent
\textbf{Step 4: Applying layercake formula}\\
Recall the following fact from measure theory:  given a measure space $(X,\mu)$ and $f:X\rightarrow\mathbb{R}^+$ Borel measurable there is the layer-cake representation for the integral: 
\begin{equation}\label{measuretheory}
\int_Xf(x)d\mu = \int_0^{\infty}\mu(f^{-1}([t,\infty))dt.
\end{equation}
Setting $X=W_j$, $\mu(Y)=\int_{W_j\cap Y}|A|^2$ and $f=\psi$ in \eqref{measuretheory} we obtain:
\begin{equation}\label{layercake}
\int_{W_j}\psi |A|^2 = \int_0^1\int_{W_j\cap\left\{\psi>t\right\}}|A|^2 d\mathcal{H} dt.
\end{equation} 
\noindent
Combining \eqref{first} and \eqref{layercake} we obtain
\begin{equation}\label{absorb}
\int_{W_j}\psi |A|^2\leq C+\int_0^1\int_{\left\{\psi =t\right\}\cap W_j}|k_M| d\mathcal{H} dt.
\end{equation}
 Since we have a double integral on both sides of \eqref{absorb}, the two sides are comparable and the strategy now is to use our test function $\psi$ to absorb the curvature term on the RHS of \eqref{absorb} on the LHS. We thus must relate $k_M$ to the second fundamental form $A_{W_j}$. For this we need a lemma:
\begin{lemma}[cf.\ Lemma 4 \cite{i}]\label{intersection}
If $\Sigma_1$ and $\Sigma_2$ are two surfaces in a 3-manifold $M$ intersecting transversally in a curve $\gamma(t)$ parameterized by unit speed, then 
\begin{equation}
|k_M|(\gamma(t))\leq\frac{|A_{\Sigma_1}|(\gamma(t))+|A_{\Sigma_2}|(\gamma(t))}{|\sin\alpha|},
\end{equation}
where $\alpha$ denotes the angle between the normal to $\Sigma_1$ and $\Sigma_2$ at $\gamma(t)$.
\end{lemma}
\begin{proof}
By definition, $|k_M|(\gamma(t))=|\nabla^M_{\dot{\gamma}(t)}\dot{\gamma}(t)|$. Let $n_1$ denote the normal to $\Sigma_1$ and $n_2$ the normal to $\Sigma_2$. By definition of the second fundamental form, we have for $i = 1,2$ 
\begin{equation}\label{secdef}
 A_{\Sigma_i}(\dot{\gamma}(t),\dot{\gamma}(t)) = \langle\nabla_{\dot{\gamma}(t)}\dot{\gamma}(t),n_i\rangle. 
\end{equation} 
The vectors $n_1$, $n_2$ and $\dot{\gamma}(t)$ form a basis for the tangent space of $M$ at $\gamma(t)$.  Indeed, $n_1$ and $n_2$ are linearly independent since $\Sigma_1$ and $\Sigma_2$ intersect transversally, and $\dot{\gamma}(t)$ has no projection onto either $n_1$ and $n_2$ since the curve $\gamma(t)$ is contained in both surfaces.   Since $\gamma(t)$ has unit speed, $\nabla^M_{\dot{\gamma}(t)}\dot{\gamma}(t)$ has no projection onto $\dot{\gamma}(t)$ and we can express $\nabla^M_{\dot{\gamma}(t)}\dot{\gamma}(t)$ as a linear combination of $n_1$ and $n_2$.  Since $\langle n_1,n_2\rangle =\cos\alpha$, the vectors $n_1$ and $\frac{n_2-\langle n_1,n_2\rangle n_1}{\sin(\alpha)}$ are an orthonormal basis for the span of $n_1$ and $n_2$.  Projecting $\nabla^M_{\dot{\gamma}(t)}\dot{\gamma}(t)$ onto this basis we obtain in light of \eqref{secdef} (and writing $A_{\Sigma_i}$ for $A_{\Sigma_i}(\dot{\gamma}(t),\dot{\gamma}(t)))$
\begin{equation}
|k_M|^2 =A^2_{\Sigma_1}+ \frac{(A_{\Sigma_2} - (\cos\alpha) A_{\Sigma_1})^2}{\sin^2\alpha}\leq\frac{(|A_{\Sigma_1}|+|A_{\Sigma_2}|)^2}{\sin^2\alpha}.
\end{equation}
\end{proof}

\textbf{Step 5: Applying the coarea formula and using the good property of cutoff function}\\
\noindent
By the coarea formula we can rewrite the term on the RHS of \eqref{absorb} as follows:
\begin{equation}\label{coarea1}
\int_0^1\int_{\left\{\psi =t\right\}\cap W_j}|k_M| ds dt = \int_{W_j} |\nabla^{W_j}\psi||k_M|d\mathcal{H}.
\end{equation}

In order to compute $ |\nabla^{W_j}\psi|$, fix a point $x\in W_j\cap\left\{\psi =t\right\}$.    Since $K_t = \left\{\psi = t\right\}$ is a level set for $\psi$, we have $\nabla^M\psi = |\nabla^M\psi|n_{K_t}$.  Thus $\langle\nabla^M\psi,n_{W_j}\rangle =  |\nabla^M\psi|\langle n_{K_t}, n_{W_j}\rangle = |\nabla^M\psi||\cos\alpha|$.  Since $\nabla^{W_j}\psi$ is a projection onto $W_j$, we thus obtain:

\begin{equation}\label{forcoarea}
 |\nabla^{W_j}\psi| = |\nabla^M\psi||\sin\alpha|.
\end{equation}
\noindent
Plugging \eqref{forcoarea} back into \eqref{coarea1} and using \eqref{absorb} we obtain

\begin{equation}\label{as}
\int_{W_j}\psi |A_{W_j}|^2\leq C+ C\int_{W_j}|k_M||\nabla^M\psi||\sin\alpha|.
\end{equation}
\noindent
By Lemma \ref{intersection} we obtain from \eqref{as}

\begin{equation}
\int_{W_j}\psi |A_{W_j}|^2\leq C + C\int_{W_j}(|A_{W_j}|+ |A_{K_t}|)|\nabla^M\psi|.
\end{equation}
\noindent
For some $C<\infty$, there holds
\begin{equation}
\sup_{t\in [0,1]}|A_{K_t}| \leq C,
\end{equation}
because $K_t$ is a family of smooth surfaces.  Thus we obtain
\begin{equation}\label{sinscancel}
\int_{W_j}\psi |A_{W_j}|^2\leq C + C\int_{W_j}|A_{W_j}||\nabla^M\psi|+C\int_{W_j}|\nabla^M\psi|.
\end{equation}
\noindent
Multiplying numerator and denominator in the integral terms in \eqref{sinscancel} by $\sqrt{\psi}$ and applying Cauchy Schwartz as well as the bound from Step 1):
\begin{equation}
\frac{|\nabla^M\psi|^2}{\psi}\leq C,
\end{equation}
we obtain for any $\epsilon >0$, a constant $C(\epsilon)$ so that

\begin{equation}\label{finally}
\int_{W_j}\psi |A_{W_j}|^2\leq C(\epsilon)+ C\epsilon\int_{W_j}\psi|A_{W_j}|^2+C\mathcal{H}^2(W_j\cap R_{\theta_0,\phi_0}).
\end{equation}
\noindent
Taking $\epsilon$ sufficiently small in \eqref{finally}, we can move the integral term on the RHS of \eqref{finally} to the LHS and obtain (since also $\mathcal{H}^2(W_j\cap R_{\theta_0,\phi_0})$ is bounded independently of $j$),

\begin{equation}\label{oy}
(1-C\epsilon)\int_{W_j}\psi|A_{W_j}|^2\leq C(\epsilon)+C, 
\end{equation}
which completes the proof of \eqref{whatwewant}. 
\qed
\\
\\
Thus we can take $\tilde{W}:=R_{\theta_0/2, \phi_0}$, yielding \eqref{ye} and this completes the proof of Step 6.

\subsection{Completion of Step 2: $1/j$-replacements arise via surgery}\label{aftersur}
The goal of this section is to prove the following Proposition, establishing the fact that when smooth surfaces approach another surface as varifolds with multiplicity one, the genus can only drop:

\begin{prop}\label{surgeries}  Let $B$ be a sufficiently small ball in a 3-manifold and $\Sigma\subset B$ with $\partial\Sigma\subset\partial B$ ($\partial\Sigma$ possibly disconnected) and with $g(\Sigma)=g$.  Consider a minimizing sequence $\Sigma_j$ for Problem($\Sigma,Is_j(\Sigma,B))$ converging to $\Delta$ (so by Proposition 3.2 of \cite{dp}, $\Delta$ is smooth, has multiplicity one, and $\partial\Delta = \partial\Sigma$).   Then $g(\Delta)\leq g$. 
\end{prop}
\begin{rmk}
The genus of a disconnected surface is defined to be the sum of the genera of each connected component.
\end{rmk}
\noindent
\emph{Proof of Proposition \ref{surgeries}:}
By successively capping off their boundary circles in $\partial B$ with disks we can extend $\Sigma_j$ and $\Delta$ to   $\tilde{\Sigma}_j$ and $\tilde{\Delta}$ respectively, so that $\tilde{\Sigma}_j$ and $\tilde{\Delta}$ are smooth closed surfaces contained in $T_{\epsilon}(B)$ and so that for some decreasing sequence $\epsilon_j\rightarrow 0$
\begin{enumerate}
\item $\tilde{\Delta}=\Delta$ in $B$
\item $\tilde{\Sigma}_j=\Sigma_j$ in $B$
\item $\tilde{\Sigma}_j=\tilde{\Delta}$ in $T_{\epsilon}(B)\setminus T_{\epsilon_j}(B)$.
\end{enumerate}

For $\delta>0$ small enough, the nearest point projection $p:T_\delta(\tilde{\Delta})\rightarrow \tilde{\Delta}$ is well-defined.

Using the varifold convergence of $\tilde{\Sigma}_j$ to $\Delta$, and the co-area formula, Proposition 2.3 in \cite{dp} applies verbatim (we give the full details for an analagous argument in the proof of Theorem \ref{main}) to show that one can perform finitely many neck-pinch surgeries supported within $B$ on each $\tilde{\Sigma}_j$ to arrive at smooth surfaces $\hat{\Sigma}_j$ such that 
\begin{enumerate}
\item $\hat{\Sigma}_j\subset T_{\delta}(\tilde{\Delta})$
\item $\hat{\Sigma}_j = \Sigma_j$ in $T_{\delta/2}(\tilde{\Delta})$
\end{enumerate}

Since $\hat{\Sigma}_j$ arises from $\tilde{\Sigma}_j$ after neck-pinch surgeries, it follows that $g(\hat{\Sigma}_j)\leq g(\tilde{\Sigma}_j)= g(\Sigma)$.  To see that $g(\Delta)\leq g$, suppose $g(\Delta)>g$ and let $\Delta^1$ be some component of $\tilde{\Delta}$ with $g(\Delta^1)> g(\hat{\Sigma}_j^1)$, where $\hat{\Sigma}_j^1$ denotes the component of $\hat{\Sigma}_j$ that has the same boundary as $\Delta^1$.  For $j$ large, the nearest-point projection map $p$ restricted to $\hat{\Sigma}_j^1$ induces a map $\tilde{p}$ from a surface of lower genus to one of higher genus.  But every $C^1$ map of a surface of lower genus to higher genus has degree $0$, and thus the degree of $\tilde{p}$ is $0$.   On the other hand, Since $\hat{\Sigma}_j^1$ and $\Delta^1$ coincide in an open set (i.e. in $T_{\epsilon}(B)\setminus T_{\epsilon_j}(B)$) it follows that the mod $2$ degree of $\tilde{p}$ is odd.  This is a contradiction. It follows that $g(\Delta)\leq g$.\qed.
\\

To complete Step 2, apply Proposition \ref{surgeries} with $\Sigma=\Sigma_j$ and with the minimizing sequence approaching $V_j$.  Since by item (6) in the Setup, $\Sigma_j$ has genus $0$ in each ball comprising $B_E$, it follows that the genus of $V_j$ is also zero and that $\Sigma_j$ and $V_j$ are isotopic.  From the proof of Proposition \ref{surgeries} we also obtain that $V_j$ is achieved from $\Sigma_j$ after surgeries. 

\subsection{Completion of Step 7:}\label{wh}
In this section we explain the necessary modifications of some classical results to obtain

\begin{lemma}\label{white}
There exists a finite set of points (potentially empty) $\mathcal{S}\subset R_{\theta_0/4,\phi_0}$ so that after passing to a subsequence (again not relabeled), the convergence $W_i\rightarrow n\Gamma$ in $R_{\theta_0/4,\phi_0}$ is smooth on compact subsets of $\Gamma\setminus\mathcal{S}$.
\end{lemma}

If the curve $\gamma$ happened to intersect $\mathcal{S}$, we can perturb it slightly to avoid this set, completing Step 7.

Lemma \ref{white} is an extension of the result of Choi-Schoen \cite{cs} \emph{up to the boundary} and follows from a theorem of Brian White (Theorem 3 in \cite{W}) which asserts the following (which we restate slightly in a form more convenient for us):
\begin{thm}\label{white2}
Let $M_i$ be a sequence of minimal surfaces with uniformly bounded areas contained in $N$ with boundaries converging in the Hausdorff topology to $\Lambda$.
Suppose further
\begin{equation}
\sup_j\int_{M_j}|A|^2 \leq C.
\end{equation}
Then there exists a finite set $S\subset N$ so that $M_i$ converges on compact subsets of $\Omega=N\setminus(\Lambda\cup S)$ to a minimal surface (potentially with multiplicity). Moreover, if a portion $G_i\subset\partial M_i$ has $||G_i ||_{C^{2,\alpha}}$ uniformly bounded, and $G_i\rightarrow G\subset\Lambda$, and $B_i:=\partial M_i\setminus G_i$ satisfies $B_i\rightarrow B\subset\Lambda$, then we can let $\Omega = N\setminus (B\cup S)$.
\end{thm}

Set $U=R_{\theta_0/2,\phi_0}$, and $M_i=W_i$ in Theorem \ref{white2}.  By Remark \ref{cinf}, the $C^{2,\alpha}$ norm of the curves $\alpha_i$ and $\beta_i$ making up the intersection of $W_i$ with the pieces of $\partial R_{\theta_0/2,\phi}$ at radius $r$ and $(1-\delta) r$ are all uniformly bounded.  Thus we can set $G_i=\cup_{j=1}^k(\alpha_j\cup\beta_j)$.  Thus applying Theorem \ref{white2} we obtain Lemma \ref{white}.  It follows that the surface $W_i$ extends as $n$ graphical sheets above $\Gamma$,  each of which connects $\alpha_i$ for some $i$ to the corresponding $\beta_i$.

\subsection{Completion of Step 9: Lifting multiple curves}\label{multiple}
Let us explain finally the modifications necessary for lifting multiple curves intersecting at a point $p$.  After perturbing the curves slightly we can ensure that they are disjoint from $\mathcal{B}$.  As in the statement of Proposition \ref{improvedlifting}, we can suppose all $\gamma_i$ intersect transversally at $p\in\Gamma$ and have no other point of pairwise intersection.  Using the argument in the Setup, for $\gamma_1$, a small number $\rho_1$ can be chosen and ``small balls" of radius $3\rho_1/4$ as well as ``large balls" of radius $15\rho_1/8$ along the curve $\gamma_1$.  The small balls are either even or odd.   We can further choose the points and labeling so that $B_1$ (as defined in the Setup) is the ball of radius $3\rho_1/4$ based at $p$.  Items (1)-(6) in the Setup then hold for these balls associated to $\gamma_1$.

The potential difficulty if we try to do this for \emph{each} curve $\gamma_i$ is that the small and large balls associated to each $\gamma_i$ could intersect in a complicated way near $p$ and it would then be more involved to run the alternating scheme that we employed for one curve.  The trick to avoid this difficulty is to choose $\rho_i$ for $i=\{2,..,k\}$ to be much smaller than $\rho_1$ and to be slightly careful about where some of the even and odd balls associated to these curves are placed.  Let us give more details.

By choosing $\rho_1$ small enough, we can arrange that each of the curves $\{\gamma_i\}_{i=2}^k$ intersects $\partial B_1=\partial B_{3\rho_{1}/4}(p)$ in two distinct points, $a_i$ and $b_i$, where $a_i\neq a_j$ and $b_i\neq b_j$ for $i\neq j$.  Likewise, each $\{\gamma_i\}_{i=2}^k$ intersects $\partial B_{15\rho_1/8}(p)$ in two distinct points, $c_i$ and $d_i$ with $c_i\neq c_j$ and $d_i\neq d_j$ for $i\neq j$.  Moreoever, by choosing $\rho_1$ small enough, for each $i\in\{2,..,k\}$ we can ensure
\begin{enumerate}[label=(\alph*)]
\item $\gamma_i$ is disjoint from all smaller balls associated to $\gamma_1$ aside from $B_1(=B_{3\rho_1/4}(p))$.
\item $\gamma_i$ is disjoint from all larger balls associated to $\gamma_1$ except for $B_{15\rho_1/8}(p)$.
\end{enumerate}

Let us assume such a $\rho_1$ has been fixed.  For each $i=\{2,...,k\}$ as in the Setup we associate ``small" and ``large" balls of radius $(3/4)\rho_i$ and $(15/8)\rho_i$ respectively, where we will choose $\rho_i$ to be sufficiently small.  These balls satisfy items (1)-(6) in the Setup.  By choosing $\rho_i$ small enough we can ensure that statements (a) and (b) hold where the symbol ``$\gamma_i$" replaced by the ``union of the larger balls associated to $\gamma_i$".  Let us refer to these amended statements as (a\textprime) and (b\textprime).  By choosing $\rho_i$ small enough we can also ensure (c\textprime): for $i, j\in\{2,...k\}$, the large balls associated to $\gamma_i$ are disjoint from the large balls associated to $\gamma_j$ when $i\neq j$.

Moreover, we choose $\rho_i$ and the small balls for each $\gamma_i$ so that for each $i=\{2,...k\}$
\begin{enumerate}
\item the two small balls associated to $\gamma_i$ that contain $a_i$ and $b_i$ are even
\item the points of intersection of $\partial B_{15\rho_1/8}(p)$ with $\gamma_i$ (i.e.\ $c_i$ and $d_i$) are contained in one of the odd small balls associated to $\gamma_i$. 
\end{enumerate}

For each $i=\{2,...k\}$, we \emph{discard} those small balls along $\gamma_i$ that are contained entirely in $B_1$.  In this way, for each $i$, the remaining small balls associated to $\gamma_i$ together with $B_1$ give a covering of $\gamma_i$.

As in the case of one curve, first perform the $1/j$-replacement $V_j$ in all even small balls for all $\gamma_i$ ($i\in\{1,2,...k\}$) simultaneously, and then in all such odd small balls.  By (b\textprime) and (c\textprime), the only thing we have to check is that this is a coherent operation for all large balls that intersect the larger ball $B_{15\rho_1/8}(p)$ associated to $\gamma_1$.

By (2), all the small even balls intersecting $B_{15\rho_1/8}(p)$ are contained entirely inside $B_{15\rho_1/8}(p)$.   Moreover, by (a\textprime) and (c\textprime), all such small even balls intersecting $B_{15\rho_1/8}(p)$ are pairwise disjoint so that when we take the $1/j$ replacement in each such ball, by applying Lemma \ref{stillam2} the resulting sequence $V_j$ is still $1/j$-almost minimizing in $B_{15\rho_1/8}(p)$, and therefore in each of the smaller odd balls associated to any of the curves $\gamma_i$ for $i\in\{1,...,k\}$ that are contained entirely in $B_{15\rho_1/8}(p)$.  

We can thus take the $1/j$-replacement as before in all small odd balls associated to each $\gamma_i$ (for all $i\in\{1,2,...,k\}$).   There are finally the small odd balls as in (2) that are intersecting, but not contained entirely in $B_{15\rho_1/8}(p)$ (containing $c_i$ and $d_i$).  These balls however are contained in some larger ball associated to a curve $\gamma_i$ for $i\in\{2,...,k\}$, which by (b\textprime) and (c\textprime) is disjoint from the other large balls associated to the other curves.   The sequence $V_j$ is thus still $1/j$-almost minimizing in these odd balls by item (4) from the Setup.

In this way we can then apply the argument in the case of one curve to show that we can lift each $\gamma_i$ on $n$ graphical levels past the overlapping regions of the small even and odd balls.
\section{Proof of Theorem \ref{main}}\label{last}
In this section, following Frohman-Hass \cite{fh}, we use the Improved Lifting Lemma (Proposition \ref{improvedlifting}) to prove Theorem \ref{main}.

Let us first recall some notation.  For $\epsilon>0$ small enough consider the tubular neighborhood $T_\epsilon(\Gamma)$ about a closed surface $\Gamma$.   If $\Gamma$ is orientable, then $T_\epsilon(\Gamma)$ is diffeomorphic to $\Gamma\times[-\epsilon,\epsilon]$ and $\partial T_\epsilon$ is diffeomorphic to two copies of $\Gamma$. If $\Gamma$ is non-orientable, $T_\epsilon(\Gamma)$ is diffeomorphic to a twisted interval bundle over $\Gamma$ and $\partial T_\epsilon(\Gamma)$ is a double cover of $\Gamma$.

To prove Theorem \ref{main} we will exploit the following basic topological fact about surfaces (see Section 1.5 in \cite{m}).  It will allow us to reduce the global question of the how min-max sequences converge to studying the min-max sequence in the neighborhood of finitely many appropriately chosen curves on the limit:
\begin{lemma}\label{top}
If $\Sigma$ is a closed surface (orientable or non-orientable) not homeomorphic to the two-sphere  and $p\in\Sigma$, then there exist simple closed curves $\gamma_1,...\gamma_k$ so that $\Sigma\setminus\{\gamma_1,...\gamma_k\}$ is homeomorphic to an open disk.  Moreoever $\gamma_i\cap\gamma_j=p$ for all distinct $i$ and $j$. If $\Sigma$ has genus $g$ and is orientable, then $k=2g$.  If $\Sigma$ has genus $g$ and is non-orientable, then $k=g$.
\end{lemma}
\begin{proof}
An orientable surface of genus $g$ that is not equal to $0$ can be represented as a regular $4g$-gon with edges identified.  After identifications each such edge becomes a closed curve, each of which intersects every other edge in only one point.  Similarly, a non-orientable surface of genus $g$ can be represented as a $2g$-gon with appropriate identifications (cf.\ Section 1.5 \cite{m}).
\end{proof}
We will also need the following lemma:
\begin{lemma}[Lemma C.1 in \cite{dp}]\label{fur}
Suppose $\Sigma\subset U$ is a surface with $\partial\Sigma\subset\partial U$, and so that $\partial\Sigma$ intersects $\partial U$ transversally in a set of circles.  Then there exists a sequence of isotopies $\phi_l$ supported in $U$ so that $\phi_l(1,\Sigma)\rightarrow\tilde{\Sigma}$, where $\tilde{\Sigma}$ arises topologically from $\Sigma$ via finitely many neck-pinches on $\Sigma$, $\partial\Sigma=\partial\tilde{\Sigma}$ and $\tilde{\Sigma}$ consists of a union of disks.
\end{lemma}
\noindent
\emph{Proof of Theorem \ref{main}.}  
Let $\Gamma$ denote the support of the min-max limit $\sum_{i=1}^kn_i\Gamma_i$.  We can assume without loss of generality that $\Gamma$ is connected (otherwise we repeat the following argument for each connected component).  Furthermore assume $\Gamma$ is not homeomorphic to the two sphere (we will indicate the changes in this case at the end of the proof).  Suppose $\Gamma$ occurs as the min-max limit with multiplicity $n$.

Fix $\epsilon_1$ smaller than the $\epsilon_0$ provided by the Improved Lifting Lemma (Proposition \ref{improvedlifting}) and let $\epsilon< \epsilon_1$.  Consider the projection map $p: T_{\epsilon}(\Gamma)\rightarrow\Gamma$.  Fix a set of curves $\left\{\alpha_l\right\}_{l=1}^k\subset\Gamma$ provided by Lemma \ref{top} so that $\Gamma_i\setminus\cup_{l=1}^k\alpha_l$ is a disk.   By our Improved Lifting Lemma (Proposition \ref{improvedlifting}) we can perturb the curves $\alpha_l$ slightly, pass to a subsequence (not relabeled) and perform finitely many surgeries on $\Sigma_j$ to obtain a new sequence $\tilde{\Sigma}_j$ (with $\tilde{\Sigma}_j\rightarrow n\Gamma$) so that if $\Gamma$ is orientable, for each $l$ there are $n$ curves $\{\tilde{\alpha}_{ls}\}_{s=1}^{n}$ in $\tilde{\Sigma}_j\cap T_\epsilon(\Gamma)$ that for each $s$, $\alpha_{ls}$ projects via $p$ with degree $1$  onto $\alpha_l$.  If $\Gamma$ is non-orientable, then we instead obtain $n/2$ curves $\{\tilde{\alpha}_{ls}\}_{s=1}^{n/2}$ in $\tilde{\Sigma}_j\cap T_\epsilon(\Gamma)$ with the property that for each $s$, $\alpha_{ls}$ projects via $p$ with degree $2$ onto $\alpha_l$.
Moreoever, for each $l$, 
\begin{equation}\label{else}
p^{-1}(\alpha_{l})\cap T_\epsilon(\Gamma)\cap\tilde{\Sigma}_j=\cup_{s=1}^{n^*}\alpha_{ls},
\end{equation}
\noindent
where $n^*$ denotes $n$ if $\Gamma$ is orientable, and $n/2$ if $\Gamma$ is non-orientable.

We will now perform surgeries on $\tilde{\Sigma}_j$ so that the resulting sequence is contained in $T_{\epsilon}(\Gamma)$ (cf.\ Proposition 2.3 \cite{dp}).

Set $\Lambda = T_{\epsilon}(\Gamma)\setminus T_{\epsilon/2}(\Gamma)$.  Since $\tilde{\Sigma}_j\rightarrow n\Gamma$, for any $\eta > 0$ and for $j$ large enough we obtain that 
\begin{equation}
\mathcal{H}^2(\tilde{\Sigma}_j\cap\Lambda)\leq\eta.
\end{equation}
For $\sigma\in[\epsilon/2,\epsilon]$ set $\Lambda_\sigma = \partial (T_{\sigma}(\Gamma))$.  By the coarea formula, 
\begin{equation}
\int_{\epsilon/2}^{\epsilon}\mathcal{H}^1(\tilde{\Sigma}_j\cap\Lambda_\sigma)d\sigma\leq\mathcal{H}^2(\tilde{\Sigma}_j\cap\Lambda)\leq\eta.
\end{equation}
Therefore by Chebyshev's inequality, for a set of $\sigma$'s of measure at least $\epsilon/2$ we obtain for each $j$
\begin{equation}\label{smalllength}
\mathcal{H}^1(\tilde{\Sigma}_j\cap\Lambda_\sigma)\leq\frac{2\eta}{\epsilon}.
\end{equation} 

By Sard's lemma we can then choose $\sigma\in[\epsilon/2,\epsilon]$ so that all $\tilde{\Sigma}_j$ intersect $\partial (T_\sigma(\Gamma))$ transversally and \eqref{smalllength} holds.  By \eqref{smalllength}, if we choose $\eta$ appropriately small, $\tilde{\Sigma}_j\cap\partial (T_\sigma(\Gamma))$ consists of small circles bounding disks in $\Lambda_\sigma$ whose diameters approach $0$ as $\eta\rightarrow 0$.  Thus for $\delta>0$ suitably small, $\tilde{\Sigma}$ intersects $T_{\epsilon+\delta}(\Gamma)\setminus T_{\epsilon-\delta}(\Gamma)$ in small cylinders whose boundary circles lie in $\Lambda_{\sigma+\delta}$ and $\Lambda_{\sigma-\delta}$.  We surger along these circles (starting with the innermost) by gluing in the appropriate small disk in $\Lambda_{\sigma+\delta}$ and $\Lambda_{\sigma-\delta}$ and removing the cylinder between them on $\Sigma_j$.   We then discard any connected component of the surgered sequence that is not contained in $T_\epsilon(\Gamma)$. The new surgered sequence (which we do not relabel) $\tilde{\Sigma}_j$ is then contained in $T_\epsilon(\Gamma)$.  

Finally we can cut open $T_\epsilon(\Gamma)$ along the surfaces $p^{-1}(\alpha_l)\cap T_\epsilon(\Gamma)$.  Namely, we can consider $N :=T_\epsilon(\Gamma)\setminus p^{-1}(\cup_{l=1}^k\alpha_l)$ which is a $3$-manifold with boundary that is homeomorphic to a $3$-ball because $\Gamma\setminus\cup_{l=1}^k\alpha_l$ is a disk.  Using the decomposition provided by Lemma \ref{top}, we can express $N$ as $P\times [-\epsilon,\epsilon]$ where $P$ is a regular polygon and  $T_\epsilon(\Gamma)$ is obtained from $N$ by appropriate boundary identifications of the faces of $\partial P\times [-\epsilon,\epsilon]$.

By \eqref{else} and the discussion preceding it, we obtain that $\tilde{\Sigma}_j\setminus p^{-1}(\cup_{l=1}^k\alpha_l)$ is contained in $N$ and has boundary in $N$ consisting precisely of $n$ closed parallel curves in $\partial P\times [-\epsilon,\epsilon]$. 

Applying Lemma \ref{fur} with $U=N$, and $\Sigma=\tilde{\Sigma}_j$ we can surger $\tilde{\Sigma}_j$ yet again in $N$ to consist of $n$ connected disks $\{D_i\}_{i=1}^n$ each of whose boundary is a curve in $\partial N$ isotopic in $\partial P\times [-\epsilon,\epsilon]$ to $\partial P\times\{0\}$.  If $\Gamma$ is orientable, each such disk is isotopic to $\Gamma$ after making the boundary identifications of $\partial N$.  Thus after further isotopy in $N$ we obtain the desired decomposition (item (2) in the statement of Theorem \ref{main}).  If $\Gamma$ is non-orientable, then after making the boundary identifications of $N$, for each $i$, $D_i$ and $D_{n-i}$ pair off to give a double cover of $\Gamma$.  Thus similarly we obtain (2).

Finally, in the case that $\Gamma$ is homeomorphic to a two-sphere, choose any simple closed curve $\gamma\subset\Gamma$ and repeat the above argument.   The only difference occurs in the final three paragraphs: the set $\Gamma\setminus\gamma$ has two connected components in $\Gamma$, not one.  Thus $T_\epsilon(\Gamma)\setminus p^{-1}(\gamma)$ consists of two $3$-balls, $N_1$ and $N_2$, and one applies the preceding three paragraphs to $N_1$ and $N_2$ separately.
\qed


\noindent

\section{Appendix A}
Let $\gamma(s)$ be a smooth closed curve in $\mathbb{R}^3$ parameterized by arclength.  Then we define the curvature $|k_{g}|$ of $\gamma(s)$ to be $|\ddot{\gamma}(s)|$.  According to Fenchel's theorem (\cite{f}, \cite{mr}), 
\begin{equation}\label{fen}
\int_{\gamma(s)} |k_g|ds\geq 2\pi.
\end{equation}

We need a version of \eqref{fen} for curves contained in small balls in a Riemannian manifold.  For a point $x$ in a Riemannian manifold $M$, consider normal exponential coordinates about $x$ for some radius $r_0$ smaller than the injectivity radius of $M$. Consider spherical coordinates $(\rho,\theta,\phi)$ for $T_xM$ and for each $\theta_0\in[0,2\pi]$ and $r\leq r_0$ denote the planar region 
\begin{equation}
P_{\theta_0,r} =\exp_x(\left\{\theta = \theta_0\right\}\cap B(r)),
\end{equation}
where $B(r)$ is the ball of radius $r$ centered around the origin in $T_xM$.
We have the following
\begin{lemma}
 Let $M$ be a Riemannian $3$-manifold. Given $x\in M$ there exists a radius $r(x)>0$ such that for any closed smooth curve $\gamma$ contained for some $\theta_0$ in $P_{\theta_0,r(x)}$ there holds
\begin{equation}
\int_{\gamma} |k_M|ds\geq\pi.
\end{equation}
\end{lemma}

\begin{proof}  
Let $k$ be an upper bound for the sectional curvature of $M$.
Fix $r\leq r_0$.   For each $\theta_0\in[0,2\pi]$ consider the surface 
$P_{\theta_0,r}$.  The Gaussian curvature of this surface at the point $x$ is the sectional curvature of $M$ in the plane through $x$ determined by $\theta = \theta_0$ in $T_xM$.  Thus the Gaussian curvature of $P_{\theta_0,r}$ is at most $2k$ if $r$ is chosen small enough.  Since $[0,2\pi]$ is compact, we can choose $r$ small enough so that for each $\theta\in[0,2\pi]$, the Gaussian curvature of $P_{\theta,r}$ is at most $2k$.   

Consider a closed curve $\gamma$ in $P_{\theta_0,r}$ bounding a region $\Gamma$.  By the Gauss-Bonnet formula,
\begin{equation}\label{ax}
2\pi = \int_\Gamma K_\Gamma  d\mathcal{H}+\int_{\gamma}k_g ds.
\end{equation}
But 
\begin{equation}\label{ax1}
|\int_\Gamma K_\Gamma d\mathcal{H}|\leq 2k\mathcal{H}^2(P_{\theta_0,r})
\end{equation}
and 
\begin{equation}\label{ax2}
|\int_\gamma k_g ds|\leq \int_\gamma |k_M|ds,
\end{equation}
where $k_M$ is the ambient curvature of the curve $\gamma$.  By decreasing $r$ if necessary we can guarantee that \begin{equation}\label{ax3}
\mathcal{H}^2(P_{\theta_0,r})\leq\frac{\pi}{2k}\mbox{               for all  } \theta_0\in[0,2\pi].
\end{equation}
Combining \eqref{ax}, \eqref{ax1}, \eqref{ax2} and \eqref{ax3} we obtain
 \begin{equation}
\int_\gamma |k_M|ds\geq\pi.
\end{equation}
\end{proof}


\begin{thebibliography}{CaGo}
\bibitem[A]{a}
F.J. Almgren Jr. \emph{The theory of varifolds}. Mimeographed notes, Princeton University, 1965.
\bibitem[B]{B}
G.D. Birkhoff. \emph{Dynamical systems with two degrees of freedom,} Trans. Amer. Math. Soc. 18 (1917) 199-300.
\bibitem[CG]{CG}
A.J. Casson and C.M. Gordon. \emph{Reducing Heegaard splittings}. Topology Appl. 27 (1987)
275--283.
\bibitem[CS]{cs}
H.I. Choi and R. Schoen. \emph{The space of minimal embeddings of a surface into a 3-manifold with positive Ricci curvature}. Invent. Math., vol. 81, (1985) 357--394.
\bibitem[CD]{cd} T.H. Colding, C. De Lellis. \emph{The min-max construction of minimal surfaces}. Surv. Differ. Geom., VIII; p. 75--107. Int.
Press, Somerville, MA, 2003.
\bibitem[CGK]{CGK}
T.H. Colding, D. Gabai and D. Ketover. \emph{On the classification of Heegaard splittings}, arxiv.org/abs/1509.05945.
\bibitem[CM]{cmestimates}
T.H. Colding and W. Minicozzi II. \emph{Estimates for parametric elliptic integrands}.  Int Math Res Notices vol. 6 (2002) 291--297.
\bibitem[CM2]{cm3}
T.H. Colding and W. Minicozzi II. \emph{A Course in Minimal Surfaces}. GSM 121, American Mathematical Society, Providence Rhode Island, 2011.
\bibitem[CM3]{cm2}
T.H. Colding and W. Minicozzi II. \emph{The space of embedded minimal surfaces of fixed genus in 3-manifold, III: Planar Domains}. Ann. of Math. (2) 160 (2004) no. 2. 523--572.
\bibitem[DP]{dp}
C. De Lellis and F. Pellandini. \emph{Genus Bounds for Minimal Surfaces Arising from the Min-Max Construction}.  J. Reine Angew. Math 644 (2010), 47-99.
\bibitem[F]{f}
W. Fenchel. \emph{\"Uber die Krummung und Windung geschlossener Raumkurven,} Math. Ann. 101 (1929) 238-252.
\bibitem[FH]{fh}
C. Frohman and J. Hass. \emph{Unstable minimal surfaces and Heegaard splittings}, Invent. Math., vol 95 no 3 (1989) 529--540.
\bibitem[I]{i}
T. Ilmanen. \emph{Lectures on Mean Curvature Flow and Related Equations} in Conference on Partial Differential Equations and Applications to Geometry, 21 August - 1 September, 1995, ICTP, Trieste.
\bibitem[G]{g}
M. Grayson. \emph{Shortening embedded curves},  Ann. Math. 120 (1989) 71-112.
\bibitem[K]{K}
D. Ketover. \emph{Equivariant min-max theory}, preprint.
\bibitem[KMN]{KMN}
D. Ketover, F.C. Marques and A. Neves. \emph{Catenoid estimate and its geometric applications}, arxiv.org/abs/1601.04514.


\bibitem[MN]{mn}
F.C. Marques and A. Neves. \emph{Min-max theory and the Willmore Conjecture}, Ann. of Math. vol 179 no 2 (2014) 683--782.
\bibitem[MN2]{mn2}
F.C. Marques and A. Neves. \emph{Morse index and multiplicity of min-max minimal hypersurfaces}, arxiv.org/abs/1512.06460.
\bibitem[M]{m}
W. Massey. \emph{Algebraic topology: an introduction}, Graduate Texts in Mathematics vol. 56, Springer NY, 1997.
\bibitem[MSY]{msy}
W. Meeks, III, L. Simon, and S.T. Yau. \emph{Embedded minimal surfaces, exotic spheres, and
manifolds with positive Ricci curvature}. Ann. of Math. (2) 116 (1982) 621–-659.
\bibitem[MR]{mr}
S. Montel and A. Ros. \emph{Curves and Surfaces} Graduate Studies in Mathematics vol 69, AMS 2009.
 \bibitem[P]{p}
J.T. Pitts. \emph{Existence and regularity of minimal surfaces on Riemannian manifolds}.
Mathematical Notes, vol 27. Princeton University Press, Princeton, N.J., 1981.
\bibitem[PR]{pr}
J.T. Pitts and J.H. Rubinstein. \emph{Existence of minimal surfaces of bounded topological type in threemanifolds},
In Miniconference on Geometry and Partial Differential Equations (Canberra, 1985), vol. 10
of Proc. Centre Math. Anal. Austral. Nat. Univ; 163--176. Austral. Nat. Univ., Canberra, 1986.
\bibitem[PR2]{pr2}
J.T. Pitts and J.H. Rubinstein. \emph{Applications of minimax to minimal surfaces and the topology of
3-manifolds}. In Miniconference on Geometry and Partial Differential Equations, 2 (Canberra, 1986), vol.
12 of Proc. Centre Math. Anal. Austral. Nat. Univ., p. 137--170. Austral. Nat. Univ., Canberra,
1987.
\bibitem[R]{r}
J.H. Rubinstein. \emph{Minimal Surfaces in Geometric 3-manifolds}. preprint, 2004, ms.unimelb.edu.au/\textasciitilde rubin/publications/minimalsurfacenotes8.pdf.
\bibitem[S]{schoen}
   R. Schoen.
   \emph{Estimates for stable minimal surfaces in three-dimensional manifolds}. In Seminar on
Minimal Submanifolds, vol 103 of Ann. of Math. Stud. p. 111--126, Princeton University Press, Princeton, NJ 1983.
\bibitem[Sc]{Sc}
H.A. Schwarz, \emph{\"Uber einen Grenz\"ubergang durch alternierendes Verfahren} Vierteljahrsschrift der Naturforschenden Gesellschaft in Z\"urich, vol 15 (1870), 272–-286.


\bibitem[SS]{ss}
F. Smith. \emph{On the existence of embedded minimal two spheres in the three sphere, endowed with an arbitrary
riemannian metric}. PhD thesis, Supervisor: Leon Simon, University of Melbourne, 1982.
\bibitem[W]{W}
B. White. \emph{Curvature estimates and compactness theorems in $3$-manifolds for surfaces that are stationary for parametric elliptic functionals,} Invent. Math. vol 88 (1987), no. 2., 243-256.







\end{thebibliography}
\end{document}